\documentclass[12pt]{article}
\usepackage[latin1]{inputenc}
\usepackage[english]{babel}
\usepackage{amsmath,amssymb,amsthm}
\usepackage[dvips]{graphicx}
\usepackage{graphicx}
\usepackage{fullpage}
\usepackage{float}
\usepackage{varioref}
\usepackage[active]{srcltx}
\usepackage{mathrsfs}
\usepackage{authblk}
\usepackage{xcolor}
\usepackage{epstopdf}
\usepackage{color}
\usepackage{caption}
\usepackage{multirow}
\usepackage{enumerate}
\usepackage{booktabs}
\usepackage{stmaryrd}
\usepackage{url}
\usepackage{graphicx}
\usepackage{epstopdf}
\usepackage{verbatim}
\usepackage{appendix}

 \usepackage{enumitem}
\usepackage{ifthen}
\usepackage[export]{adjustbox}
 \usepackage[lined,boxed,linesnumbered]{algorithm2e}
 \SetKwInOut{Parameter}{parameter}

\definecolor{refkey}{gray}{.75}
\definecolor{labelkey}{gray}{.75}
\usepackage[dvips]{graphicx}
\def\RR{\mathbb   R}
\def\EE{\mathbb   E}
\def\NN{\mathbb   N}
\def\PP{\mathbb   P}

\def\1{{\bf 1}}
\definecolor{forest}{rgb}{0.13, 0.55, 0.13}

\newcommand{\V}{\mathrm{Var}}

 \newcommand{\Et}{\tilde{\mathbb{E}}}
\newtheorem{theorem}{Theorem}[section]
\newtheorem{proposition}{Proposition}[section]
\newtheorem{lemma}{Lemma}[section]

\newtheorem{remark}{Remark}
\newcommand{\argmin}{\mathop{\mathrm{arg\,min}}}
\newcommand\independent{\protect\mathpalette{\protect\independenT}{\perp}}

\def\independenT#1#2{\mathrel{\rlap{$#1#2$}\mkern2mu{#1#2}}}

\begin{document}
\title{\textsc{Adaptive Importance Sampling for Multilevel Monte Carlo Euler method}}
\author[1]{\textsc{Mohamed Ben Alaya}\thanks{Supported by Laboratory of Excellence MME-DII http://labex-mme-dii.u-cergy.fr/}}
\author[1]{\textsc{Kaouther Hajji}}
\author[1]{\textsc{Ahmed Kebaier}\thanks{This research benefited from the support of the chair "Risques Financiers", Fondation du Risque.}\thanks{Supported by Laboratory of Excellence MME-DII http://labex-mme-dii.u-cergy.fr/}}
\affil[1]{{\normalsize{ Universit\'{e} Paris 13,
Sorbonne Paris Cit\'e, LAGA, CNRS (UMR 7539)}}
\normalsize{mba@math.univ-paris13.fr\hspace{0.2 cm}
hajji@math.univ-paris13.fr\hspace{0.2cm}kebaier@math.univ-paris13.fr}}
\maketitle

\begin{abstract}
This paper focuses on the study of an original combination of the Multilevel Monte Carlo method introduced by Giles \cite{Giles} and the popular importance sampling technique. To compute the optimal choice of the parameter involved in the importance sampling method, we rely on Robbins-Monro type stochastic algorithms.  On the one hand,  we extend our previous work  \cite{BenalayaHajjiKebaier} to the Multilevel Monte Carlo setting. On the other hand,  we improve  \cite{BenalayaHajjiKebaier} by providing a new  adaptive algorithm  avoiding the discretization of any additional process. 
Furthermore, from a technical point of view, the use of the same stochastic algorithms as in \cite{BenalayaHajjiKebaier} appears to be problematic. To overcome this issue, we employ an alternative version of stochastic algorithms with projection (see e.g.  Laruelle, Lehalle and Pag\`es \cite{LaruellepagesLehalle}).  In this setting, 
we show innovative limit theorems for a doubly indexed stochastic algorithm which appear to be crucial  to study the asymptotic behavior of the new adaptive Multilevel Monte Carlo estimator. Finally, we illustrate the efficiency of our method through applications from quantitative finance.
\smallskip
\smallskip


\end{abstract}
\section*{Introduction}
We are interested in estimating the expected payoff value $\EE \psi(X_T)$ in option pricing problems, with $T>0$ and $(X_t)_{ 0 \leq t\leq T}$ is a given diffusion model defined on $\mathcal B=(\Omega,\mathcal F,(\mathcal F_t)_{t\geq 0},\PP)$.
\\In view of reducing the variance in the estimation, we envisage using the importance sampling technique. For the Gaussian setting, 
the practical use of this idea  with Monte Carlo (MC) methods  was firstly studied by Arouna  \cite{Arouna}, Glasserman, Heidelberger and 
Shahabuddin \cite{GlassermanheidelbergerShahabuddin}. Later on, Ben Alaya, Hajji and Kebaier
\cite{BenalayaHajjiKebaier} studied the use of this procedure with the Statistical Romberg  (SR) algorithm known for reducing the computation time in the MC method. The approach used in  \cite{BenalayaHajjiKebaier}, consists on
 applying the Girsanov theorem, to obtain 
\begin{equation}\label{price}
 \quad \EE\psi (X_T)=\EE g(\theta,X_T^\theta, W_T),  \mbox{ where } g: (\theta, x,w)\in\RR^q \times \RR^d \times \RR^q \mapsto \psi(x) e^{-\theta\cdot w-\frac{1}{2} |\theta|^{2}T}\in\RR
\end{equation} 
and $(X_t^\theta)_{0 \leq t \leq T}$ is solution to 
\begin{equation}
\label{eq_Xtheta_MLMC}
 dX_{t}^{\theta}=\left(b(X_{t}^{\theta}) +\sum_{j=1}^q\theta_j \sigma_j(X_{t}^{\theta})\right) dt+\sum_{j=1}^q\sigma_j(X_{t}^{\theta}) dW^j_{t}.
\end{equation}
Note that the case $\theta=0$ corresponds to  the stochastic differential equation satisfied by $(X_t)_{0 \leq t \leq T}$.
÷rm Then, the optimal $\theta_{\rm SR}^*$ reducing the limiting variance of  the SR method is given by
$$
 \theta_{\rm SR}^*= \underset{\theta \in \RR^q}{\argmin} \left[\tilde{\V}(g(\theta,X^{\theta}_{T},W_T))+\tilde \V \left( \nabla_x g(\theta,X^{\theta}_{T},W_T)\cdot U^{\theta}_{T}  \right)\right],
$$
where  $U^{\theta}$ is is a given diffusion associated to the process $(X_{t}^{\theta})_{t \in [0,T]}$ defined on an extension $\tilde{\mathcal B}=(\tilde{\Omega},\tilde{\mathcal F},( \tilde{\mathcal F}_{t})_{t\geq 0},\tilde{\PP})$ of the initial space $\mathcal B$ (see further on).
Then, they suggest to approximate the above variance optimizer by 
\begin{equation}\label{var:SR}
\theta_{\rm SR}^{*,n}=\arg\min_{\theta\in\RR^q}\tilde \EE\left( g^2(\theta,X^{n,\theta}_T,W_T)+(\nabla_xg(\theta,X^{n,\theta}_T,W_T)\cdot U^{n,\theta}_T)^2\right),
\end{equation}
where $ X^{n,\theta}_T$ (resp. $U^{n,\theta}_T$) is the Euler scheme, with time step $T/n$, associated to $X^{\theta}_T$ (resp. $U^{\theta}_T$).

The aim of this paper is to combine the Multilevel Monte Carlo (MLMC)  with the importance sampling technique  for a properly
chosen optimal $\theta$ based on stochastic algorithms. The MLMC method was introduced and popularized for financial applications by  Giles \cite{Giles}. This method can be seen as a generalized version of SR method introduced by Kebaier in \cite{Keb}.
It has been extensively applied to various fields of numerical probability (Brownian stochastic differential equations, L\'evy-driven stochastic differential equations and
more general numerical analysis problems, see e.g. \cite{KBA,Collier,DerLi,GilesSzpruch,Hoel}). For more references, we refer to the web page \url{https://people.maths.ox.ac.uk/gilesm/mlmc_community.html} and the references therein. 
Note that Jourdain and Lelong \cite{JouLel} introduced another approach based on a deterministic minimization of the empirical estimation of the variance $ \V(g(\theta,X^{\theta}_{T},W_T))$.   This approach was recently used by Kebaier and Lelong \cite{ KebLel} to the MLMC methods combined with importance sampling technique. In these two last references, the authors used deterministic optimization algorithm instead of stochastic algorithms.

In the context of Euler discretization scheme, the idea of the MLMC method is to apply the classical MC method for several nested levels of time step sizes and to compute different numbers of paths
on each level, from a few paths when the time step size is small to many paths when the step size is large. 
More precisely, the Euler MLMC method uses information from a sequence of computations with decreasing step size and approximates the quantity $\EE\psi (X_T)$ by
\begin{equation*}
 Q_n=\frac{1}{N_0} \sum_{i=1}^{N_0} \psi(X_{T,i}^{m^0}) + \sum_{\ell=1}^{L} \frac{1}{N_\ell} \sum_{i=1}^{N_\ell} \left( \psi(X_{T,i}^{\ell, m^{\ell}})-\psi(X_{T,i}^{\ell, m^{\ell-1}}) \right), \quad m \in \NN \setminus \{0,1\}.
\end{equation*}
The time step sizes is defined by $T/m^{\ell}$, $\ell=0,1,...,L$ where $L=\frac{ \log{n}}{\log{m}}$ and $m$ is the refinement factor. Here, $X_{T}^{\ell, m^\ell}$ and $X_{T}^{\ell, m^{\ell-1}}$ denote the Euler schemes of $X_T$ with time steps $ T / m^{\ell}$ and $T/ m^{\ell-1}$. 
It turns out that the optimal $\theta^*$ in this case is solution to the problem:
$
 \theta^*= \underset{\theta \in \RR^q}{\argmin} \tilde{\V} \left( \nabla_x g(\theta,X^{\theta}_{T},W_T)\cdot U^{\theta}_{T}  \right),
$
see Section \ref{Preliminaries} for more details.

In the current paper, we come up with a new idea to approximate this latter optimal $\theta^*$ through the minimization of  
\begin{equation}
 v_{\ell}(\theta):=\mathbb{E}\left[  \left( \sqrt{\frac{m^\ell}{(m-1)T}} (\psi(X_{T}^{m^\ell, \theta})-\psi(X_{T}^{m^{\ell-1}, \theta}) )\right)^2 \right],
\end{equation}
which is nothing but a proper approximation of $\tilde{\V} \left( \nabla_x g(\theta,X^{\theta}_{T},W_T)\cdot U^{\theta}_{T}  \right)$, as $\ell$ tends to infinity.
The advantage of such  approach  is that unlike \eqref{var:SR}, the new MLMC variance optimizer is approximated without any need to compute $\nabla\psi$ or to discretize the process $U$ appearing  naturally in the limit variance of the problem. 
This is clearly more convenient for practitioners since we also require less conditions on the regularity of the payoff $\psi$ and the diffusion coefficients. 

Furthermore,  Ben Alaya, Hajji and Kebaier \cite{BenalayaHajjiKebaier} consider the constrained version of Robbins Monro algorithm proposed by Chen et al. \cite{CGG, CHZY} and the unconstrained one proposed by Lemaire and Pag\`es \cite{LemPag} to recursively approximate   $\theta_{\rm SR}^{*,n}$.
From a technical point of view, the use of these two stochastic algorithms  seems to be problematic in the study of the asymptotic normality of the MLMC estimator coupled with the importance sampling technique. To overcome this difficulty, we consider an alternative version of Robbins-Monro type algorithms, namely the stochastic algorithm with projection (see \cite{LaruellepagesLehalle}). Moreover,  in \cite{BenalayaHajjiKebaier}  the proof of the central limit theorem for the adaptive SR algorithm relies on the law stable convergence theorem for the error Euler scheme obtained in Jacod and Protter \cite{JJPP} and which is given by
$$\sqrt{n} \left(X^n-X \right) \overset{\rm stably}\Longrightarrow  U, \quad \mbox{as } n\rightarrow \infty.$$
This theorem cannot be used in the MLMC setting, since we should consider the Euler error on two consecutive levels $m^{\ell-1}$ and $m^\ell$. To cope with this situation, we rather use  Theorem 3 in Ben Alaya and Kebaier \cite{KBA} given by  
$$\sqrt{\frac{m^\ell}{(m-1)T}}(X^{m^\ell}- X^{ m^{\ell-1}}) \overset{\rm stably}\Longrightarrow  U , \quad \mbox{as } \ell\rightarrow \infty.$$

In the next section, we recall some essential results on the MLMC method. In section \ref{sec:OP}, we set our optimisation problem. In section \ref{algo:sto}, we introduce a doubly indexed constrained stochastic algorithm 
and prove innovative convergence theorems on it (see Remark \ref{original:dl} below). A distinctive feature of this limit theorem is that for a fixed order in letting the indexes tend to infinity we prove almost sure convergence to the optimal parameter $\theta^*$ ( see the second assertion of 
Theorem \ref{constrained_algo_MLMC} and  Theorem \ref{th:convergence_MLMC} ). However, when reverting this order, we prove a stable law convergence towards  $\theta^*$ (see Theorem \ref{constrained_rec} and the first assertion of Theorem \ref{constrained_algo_MLMC}). To the best of our knowledge such result does not exist in the literature around stochastic algorithms and it is crucial for the study of the adaptive MLMC estimator.
Next, in section \ref{CLTadp_MLMC}, we first introduce the new adaptive algorithm obtained by combining together the importance sampling procedure and the  MLMC method.
Then, taking advantage of  Section \ref{algo:sto} results, we prove a Lindeberg-Feller central limit theorem for this new algorithm (see Theorem \ref{CLTMLMC_adaptatif2}). 
In section \ref{hestonmodel_MLMC}, we give a time complexity analysis of the adaptive MLMC algorithm and we proceed to numerical simulations to illustrate the efficiency of this new method for pricing an European call option under the Black \& Scholes model. 

\section{Preliminaries}
\label{Preliminaries}
Let $(X_t)_{ 0 \leq t\leq T}$ be the process with values in $\RR^d$, solution to the diffusion
\begin{equation}\label{1_MLMC}
dX_{t} = b(X_{t})dt + \sum_{j=1}^q {{\sigma_j(X_{t})}}dW^j_{t},\;\;\; X_0=x\in \mathbb R^d,
\end{equation}
where $\displaystyle W=(W^1,\dots,W^q)$ is a $q$-dimensional Brownian motion on some given filtered probability space
 $\mathcal B=(\Omega,\mathcal F,(\mathcal F_t)_{t\geq 0},\PP)$ and $(\mathcal F_t)_{t\geq 0}$ is the standard Brownian filtration. The functions  $\displaystyle b:{\mathbb R^d\longrightarrow \mathbb R^d}$ and  
$\displaystyle{\sigma_j:\mathbb R^d\longrightarrow \mathbb R^{d}}$, $1\leq j \leq q$,  satisfy condition
$$
\quad \forall x,y\in\mathbb R^d \quad |b(x)-b(y)|+\sum_{j=1}^q|\sigma_j(x)-\sigma_j(y)|\leq C_{b,\sigma}|x-y|,\;\;\mbox{ with } C_{b,\sigma}>0.\leqno({{\mathbf{\mathcal H}_{b,\sigma}}})
$$
This ensures strong existence and uniqueness of solution  of (\ref{1_MLMC}). 
In practice, we consider the Euler continuous approximation $X^n$ of the process $X$, with time step $\delta=T/n$ given by
\begin{equation}\label{euler_MLMC}
dX^n_t=b(X_{\eta_n(t)})dt+\sum_{j=1}^q \sigma_j(X_{\eta_n(t)})dW_t,\;\;\;\eta_n(t)=[t/\delta]\delta.
\end{equation}
It is well known that under  condition $(\mathcal H_{b,\sigma})$ we have the almost sure convergence of $X^n$ towards $X$ together with the  following property   (see e.g. Bouleau and L\'epingle \cite{BouLep}) 
 $$\label{PSR}\;\; \displaystyle{ \forall p\geq 1,\;\;\;\sup_{0 \leq t \leq T}|X_t|,\,\sup_{0 \leq t \leq T}|X^n_t|\in L^p\quad\mbox{ and }\;\; \mathbb{E}\left[{\sup_{0 \leq t \leq T}}\arrowvert X_t -  {X}^n_t\arrowvert^{p} \right] \leq\frac{K_p(T)}{n^{p/2}}},\leqno{(\bf \mathcal P)}$$
  where  $K_p(T)$ is a positive constant depending only on $b$, $\sigma$, $T$, $p$ and $q$.

The Euler MLMC method uses information from a sequence of computations with decreasing step sizes and approximates the quantity $\EE\psi (X_T)$ by
\begin{equation*}
 Q_n=\frac{1}{N_0} \sum_{i=1}^{N_0} \psi(X_{T,i}^{m^0}) + \sum_{\ell=1}^{L} \frac{1}{N_\ell} \sum_{i=1}^{N_\ell} \left( \psi(X_{T,i}^{\ell, m^{\ell}})-\psi(X_{T,i}^{\ell, m^{\ell-1}}) \right), \quad m \in \NN \setminus \{0,1\}.
\end{equation*}
The time step sizes is defined by $T/m^{\ell}$, $\ell=0,1,...,L,$ where $L=\frac{ \log{n}}{\log{m}}$ and $m$ is the refinement factor.
For the first empirical mean, the random variables $(X_{T,i}^{m^{0}})_{1 \leq i \leq N_0}$ are independent copies of $X_{T}^{m^{0}}$ which denotes the Euler scheme with time step T.
For $\ell \in \{1,...,L\}$, the couples $(X_{T,i}^{\ell, m^\ell},X_{T,i}^{\ell, m^{\ell-1}})_{1 \leq i \leq N_\ell}$ are independent copies of 
$(X_{T}^{\ell, m^\ell},X_{T}^{\ell, m^{\ell-1}})$ whose components denote the Euler schemes with time steps $ T / m^{\ell}$ and $T/ m^{\ell-1}$. However, for fixed $\ell$, the simulation of $X_{T}^{\ell, m^\ell}$ and $X_{T}^{\ell, m^{\ell-1}}$ has to be based on the same Brownian path. 
Giles \cite{Giles} proved that the MLMC method reduces efficiently the computational time complexity of the combination of Monte Carlo method and the Euler discretization scheme.  
In fact, the time complexity in the Monte Carlo method is equal to $n^{2\alpha+1}$ 
and is reduced to $n^{2\alpha} (\log{n})^{2}$ in the MLMC method where $\alpha \in [1/2,1]$ is the order of the rate of convergence of the weak error given by $\varepsilon_n= \EE \psi(X_T^n) - \EE \psi(X_T)$.
So, it is worth introducing the following assumption
\begin{equation*}
\label{Hepsilon} \quad\mbox{for } \alpha \in [1/2,1]\quad n^{\alpha}\varepsilon_n\rightarrow 
C_{\psi}, \quad C_{\psi}\in \RR.\leqno(\mathcal H_{\varepsilon_n})
\end{equation*}
In their recent work, Ben Alaya and Kebaier \cite{KBA}  proved that  for  $\mathscr C^1$  coefficients $\sigma$ and $b$   satisfying condition $({{\mathbf{\mathcal H}_{b,\sigma}}})$, 
\begin{equation}\label{CV:stable}
r_\ell (X^{m^\ell}- X^{ m^{\ell-1}})\overset{\rm stably}\Longrightarrow U\quad \mbox{with }\;r_\ell:=\sqrt{\frac{m^\ell}{(m-1)T}},
\end{equation}
where the process $U$ is  defined on $\mathcal {\tilde B}$  an extension of the initial space  $\mathcal B$ and is solution to
\begin{equation}
\label{eq_U_MLMC}
dU_t=\dot b(X_t)U_tdt+\sum_{j=1}^q\dot{\sigma}_j(X_t)U_tdW^j_t-\frac{1}{\sqrt{2}}\sum_{j,\ell=1}^q\dot{\sigma}_j(X_t){\sigma}_{\ell}(X_t)d\tilde W^{\ell j}_t,
\end{equation}
where $\tilde W$ is a $q^2$-dimensional standard Brownian motion, defined on the extension $\tilde {\mathcal B}$, independent of $W$, and $\dot b$ (respectively $(\dot\sigma_j)_{1\leq j\leq q}$) is the Jacobian matrix of $b$ (respectively $(\sigma_j)_{1\leq j\leq q})$. For a proof of \eqref{CV:stable} see Theorem 3 in \cite{KBA}.

Moreover they proved a central limit theorem for the Euler MLMC method with a rate of convergence equal to $n^{\alpha}$, $\alpha \in [1/2, 1]$ (see Theorem 4 of \cite{KBA}). In more details, 
for sample sizes $(N_\ell)_{0\leq \ell \leq L}$ of the form
\begin{equation}
\label{sample_size_MLMC}
N_\ell=\frac{n^{2 \alpha} (m-1) T}{m^\ell a_\ell} \sum_{\ell=1}^{L} a_\ell, \quad \ell \in \{0,...,L\} \mbox{ and } L=\frac{\log n}{\log m},
\end{equation}
where $(a_\ell)_{\ell \in \NN}$ is a real sequence of positive terms satisfying
$$ \label{W} \;\; \displaystyle{\lim_{L \rightarrow \infty} \sum_{\ell=1}^{L} a_\ell = \infty \mbox{ and } \lim_{L \rightarrow \infty} \frac{1}{(\sum_{\ell=1}^{L} a_\ell)^{p/2}} \sum_{\ell=1}^{L} a_\ell^{p/2}=0, \mbox{ with } p>2},\leqno{(\bf \mathcal W)} 
$$
for a function $\psi$ satisfying 
$$  \label{Hpsi} \quad\mathbb P(X_T\notin \mathcal{D}_{\bf\dot{\psi}} )=0, \mbox{ where } 
\mathcal{D}_{\bf\dot{\psi}}:=\{x\in\mathbb R^d\;|\; \psi \;\mbox{is differentiable at}\; x\},\leqno{(\mathcal H_{\psi})}
$$
and 
\begin{equation} \label{condition_dif_psi1}
 |\psi(x)-\psi(y)|\leq C(1+|x|^p+|y|^p)|x-y|,\;\;\;\mbox{for some $C,p>0$},
\end{equation}
under condition $(\mathcal H_{\varepsilon_n})$, they proved that
\begin{equation}\label{variance_limit}
n^{\alpha}\left(Q_n -\mathbb{E}\psi(X_{T})\right) \underset{n \rightarrow \infty} {\overset{\mathcal{L}}{\longrightarrow}} \mathcal{N} \left(C_{\psi}, {\tilde\V} \left( \nabla \psi(X_{T})\cdot U_{T}  \right) \right).
\end{equation}

The following lemma gives a helpful result that will be used several times in our forthcoming proofs. The proof is postponed to the appendix section. 
\begin{lemma}\label{lem:tech}
Let assumption $({{\mathbf{\mathcal H}_{b,\sigma}}})$ and \eqref{condition_dif_psi1} hold. Then,
\begin{enumerate}
 \item for all $q>1$, we have 
 \begin{equation}
\label{uniform_integrability}
  \sup_{\ell\ge 1} \mathbb E \left| r_\ell \ (\psi(X_{T}^{m^\ell})-\psi(X_{T}^{m^{\ell-1}}) ) \right|^{2q} < \infty.  
\end{equation}
\item If moreover assumption $(\mathcal H_{\psi})$ is satisfied and the coefficients $\sigma$ and $b$ are in $\mathscr C^1_b$, then
\begin{equation}
\label{cv_stable_MLMC}
 r_\ell \ \left(\psi(X_{T}^{m^\ell})-\psi(X_{T}^{m^{\ell-1}}) \right) \overset{\rm stably}\Longrightarrow \nabla \psi(X_T).U_T,  \quad \mbox{as } \ell \rightarrow \infty.
\end{equation}
\item Therefore, combining the first and the second assertions yields
\begin{equation}
\label{result1_MLMC}
 \EE \left( r_\ell \ \left(\psi(X_{T}^{m^\ell})-\psi(X_{T}^{m^{\ell-1}}) \right) \right)^{k} \underset{\ell\rightarrow\infty}{\longrightarrow} \tilde{\EE} \left( \nabla \psi(X_T).U_T \right)^{k} < \infty \mbox{ for } k \in \{1,2 \}.
\end{equation}
\end{enumerate}
\end{lemma}
\section{Optimization problem}\label{sec:OP}
The purpose of this work is to combine the importance sampling technique and the  MLMC to approximate   $\mathbb{E} \psi(X_{T})=\mathbb{E}g(\theta, X_{T}^{\theta},W_{T}), \theta \in \RR^q$, by
\begin{equation}
\frac{1}{N_0} \sum_{i=1}^{N_0}g(\theta, \hat X_{T,i}^{m^0,\theta},\hat W_{T,i})+ 
\sum_{\ell=1}^{L} \frac{1}{N_\ell} \sum_{i=1}^{N_\ell}\left(g(\theta, X_{T,i}^{\ell,m^\ell,\theta},W_{T,i}^\ell)-g(\theta, 
X_{T,i}^{\ell,m^{\ell-1},\theta},W_{T,i}^\ell)\right),
\end{equation}
where $X_T^{n,\theta}$ is the Euler scheme associated to $X_T^\theta$ (\ref{eq_Xtheta_MLMC}) with time step $T/n$ and $g(\theta,x,y)=\psi(x) e^{-\theta\cdot y-\frac{1}{2} |\theta|^{2}T}, \forall x\in\RR^d \mbox{ and } y\in\RR^q.$
According to relation \eqref{variance_limit} and under appropriate assumptions, the limit variance in this case is given by 
$
 {\tilde\V} \left( \nabla_x g(\theta,X^{\theta}_{T},W_T)\cdot U^{\theta}_{T}  \right),
$
where  $U^{\theta}$ is the weak limit process of the error $ r_{\ell} (X^{m^\ell,\theta}- X^{m^{\ell-1},\theta})$ defined on the extension $\mathcal {\tilde B}$ and solution to
\begin{equation}
\label{eq_Utheta_MLMC}
dU^{\theta}_t=\left(\dot b(X_{t}^{\theta}) +\sum_{j=1}^q\theta_j \dot\sigma_j(X_{t}^{\theta})\right)U^{\theta}_t dt+\sum_{j=1}^q\dot{\sigma}_j(X^{\theta}_t)U^{\theta}_tdW^j_t-\frac{1}{\sqrt{2}}\sum_{j,\ell=1}^q\dot{\sigma}_j(X^{\theta}_t){\sigma}_{\ell}(X^{\theta}_t)d\tilde W^{\ell j}_t.
\end{equation}
Note that, $U$ and $U^\theta$ are the same processes obtained for the asymptotic behavior of $(X^{ m^\ell} -X)$ and $(X^{m^{\ell}, \theta}-X^{\theta})$. Moreover, according to Proposition 2.1 in Kebaier \cite{Keb},  we have
${\tilde {\mathbb E} } \left( \nabla_x g(\theta,X^{\theta}_{T},W_T)\cdot U^{\theta}_{T}  \right)=0$, so using Girsanov's theorem, we get
$$
{\tilde\V} \left( \nabla_x g(\theta,X^{\theta}_{T},W_T)\cdot U^{\theta}_{T}  \right) =\tilde{\mathbb{E}}\left[\left(\nabla \psi(X_{T})\cdot U_{T} \right)^{2} e^{-\theta \cdot W_{T}+\frac{1}{2}|\theta|^{2}T} \right].
$$
\vspace{0.1cm}
To implement properly this importance sampling technique, we have to approximate the optimal  $\theta^*$ solution to 
\begin{equation}
\label{eq_variance0_MLMC}
\theta^*=\underset{\theta \in \RR^q}{\argmin}\, v(\theta):=\tilde{\mathbb{E}}\left[\left(\nabla \psi(X_{T})\cdot U_{T} \right)^{2} e^{-\theta \cdot W_{T}+\frac{1}{2}|\theta|^{2}T} \right].
\end{equation}
As expected, to ensure the existence and uniqueness of $\theta^*$, we will require more regularity assumptions.  So we introduce 
\paragraph{\normalsize{Assumption $({\mathcal R}_{\psi,a})$:}}
\begin{itemize}
\item[$\bullet$] condition ($\mathcal H_{\psi}$) is satisfied and $ \mathbb{P}((\nabla \psi(X_{T})\cdot U_T)\neq 0) >0$,
 \item[$\bullet$]     there exists $a>1$ such that $\mathbb{E}\left[|\nabla{\psi}(X_{T})|^{2a}\right]<\infty$.
\end{itemize}
Since only $\nabla\psi$ is involved in the above variance term \eqref{eq_variance0_MLMC},  the arguments of the proof of Proposition 2.1 in \cite{BenalayaHajjiKebaier} can be easily adapted to get the following result.
\begin{proposition}\label{prop:cvx}
Let assumption $({\mathcal R}_{\psi,a})$ holds and  the diffusion coefficients $\sigma$ and $b$ be in $\mathscr C^1_b$\footnote[1]{The notation $\mathscr C^1_b$ stands for  the set of $\mathscr C^1$ functions with bounded first derivative.}. Then,
the function  $\theta \mapsto v(\theta)$   is $\mathscr C^{2}$ and is strictly convex  with  $\nabla v(\theta)=\tilde \EE H(\theta,X_T,U_T,W_T)$ where
\begin{equation}
\label{eq_gradientv}
 H(\theta,X_T,U_T,W_T):= (\theta T-W_{T})(\nabla \psi(X_T)\cdot U_T)^2 e^{- \theta\cdot W_{T} + \frac{1}{2} |\theta|^{2}T}.
\end{equation}
Moreover, there exists a unique $\theta^{*} \in \mathbb{R}^{q}$ such that $\min_{\theta \in \mathbb{R}^{q}} v(\theta)=v(\theta^{*})$.
\end{proposition}
From a practical point of view it is obvious that the quantity $v(\theta)$ to be optimized is not explicit. 
At this stage, we come up with a new idea  by changing the way of approximating $\theta^*$.   More precisely,  we suggest to approximate $\theta^*$  by $\theta^*_{\ell}:=\underset{\theta \in \RR^q}{\argmin}\, v_{\ell}(\theta),$ with
\begin{multline*}
 v_{\ell}(\theta):=\mathbb{E}\left[  \left( \sqrt{\frac{m^\ell}{(m-1)T}} (\psi(X_{T}^{m^\ell, \theta})-\psi(X_{T}^{m^{\ell-1}, \theta}) )\right)^2 \right] \\
 =\mathbb{E}\left[ \left( \sqrt{\frac{m^\ell}{(m-1)T}} (\psi(X_{T}^{m^\ell})-\psi(X_{T}^{m^{\ell-1}}) ) \right)^2 e^{-\theta \cdot W_{T}+\frac{1}{2}|\theta|^{2}T} \right], \mbox{ for } \ell\geq 1.
\end{multline*}
The last relation is obtained by using a change of probability. In fact, according to Girsanov theorem, the process $(B^\theta, X)$ under $\PP^\theta$ has the same law as $(W, X^\theta)$ under $\PP$.
For $\ell=0$, we simply take
$$
v_{0}(\theta):=\mathbb{E}\left[  \psi(X_{T}^{m^0})^2 e^{-\theta \cdot W_{T}+\frac{1}{2}|\theta|^{2}T} \right].
$$
\begin{remark}\label{RM_MLMC}
The new variance associated to the MLMC is only based on the discretization of the process $X$. In fact, this is different from the problem addressed in our previous work \cite{BenalayaHajjiKebaier}, where we discretized the process $U$ in order to compute
\begin{equation*}
\theta^*_{n}:=\underset{\theta \in \RR^q}{\argmin}\, v^n_{\rm SR}(\theta), \mbox{ with }
 v^n_{\rm SR}(\theta):=\tilde{\mathbb{E}}  \left( \left[ \psi(X_{T}^{n, \theta})^2 + (\nabla \psi(X_{T}^{n,\theta}). U_T^{n,\theta} )^2 \right] e^{-\theta \cdot W_{T}+\frac{1}{2}|\theta|^{2}T} \right).
\end{equation*}
\end{remark}

But, what about the existence and uniqueness of $\theta_{\ell}^*$ ? Under what kind of assumptions does it converge to $\theta_{\ell}^*$ when $\ell\rightarrow\infty$ ? The answers are given in the two following results.
\begin{proposition}
\label{proposition1_MLMC}
 Let $\sigma$ and $b$ be in $\mathscr C^1$ and satisfying condition $(\mathcal H_{b,\sigma})$. Assume that 
 $ \mathbb{P}( (\psi(X_{T}^{m^\ell}) -\psi(X_T^{m^{\ell-1}})) \neq 0) >0$  for $\ell\geq 1$ and $ \mathbb{P}( (\psi(X_{T}^{m^0})\neq 0) >0$. 
 Moreover, let $\psi$ satisfying relation~\eqref{condition_dif_psi1}. 
 
 Then, the function  $\theta \mapsto v_{\ell}(\theta)$   is $\mathscr C^{2}$ and strictly convex  with  $\nabla v_{\ell}(\theta)=\EE H_{\ell}(\theta,X^{m^\ell}_T,X_T^{m^{\ell-1}}W_T)$ for all $\ell \in \NN $ where
\begin{equation}
\label{eq_gradientv_MLMC}
 H_{\ell}(\theta,X_T^{m^\ell},X_{T}^{m^{\ell-1}},W_T):= (\theta T-W_{T}) \left(r_\ell \ (\psi(X_{T}^{m^\ell}) - \psi(X_{T}^{m^{\ell-1}}) ) \right)^2 e^{- \theta\cdot W_{T} + \frac{1}{2} |\theta|^{2}T}, \mbox{ for } \ell\geq 1,
\end{equation}
and for $\ell=0$, we take
$
H_{0}(\theta,X_T^{m^0},X_{T}^{m^{-1}},W_T):=(\theta T-W_{T}) \psi(X_{T}^{m^0})^2e^{- \theta\cdot W_{T} + \frac{1}{2} |\theta|^{2}T}.
$
Moreover, there exists a unique $\theta_{\ell}^{*} \in \mathbb{R}^{q}$ such that $\min_{\theta \in \mathbb{R}^{q}} v_{\ell}(\theta)=v_{\ell}(\theta_{\ell}^{*})$.  
\end{proposition}
\begin{proof}
The case $\ell=0$ is treated in the same way as the case $\ell\geq 1$. So, we only give a proof for this last case. The function $\theta\mapsto \left(r_\ell (\psi(X_{T}^{m^\ell}) - \psi(X_{T}^{m^{\ell-1}}))\right)^2 e^{- \theta\cdot W_{T} + \frac{1}{2} |\theta|^{2}T}$ is infinitely continuously differentiable with a first derivative equal to 
 $H_{\ell}(\theta,X_T^{m^\ell},X_{T}^{m^{\ell-1}},W_T)$. Note that, for $c>0$ we have
\begin{equation*}
\sup_{|\theta|\leq c}| H_{\ell}(\theta,X_T^{m^\ell},X_{T}^{m^{\ell-1}},W_T)| \leq (cT+|W_{T}|) \left(r_\ell \ (\psi(X_{T}^{m^\ell}) - \psi(X_{T}^{m^{\ell-1}}))\right)^2 e^{c |W_{T}|+\frac{1}{2} c^{2}T}. 
\end{equation*}
Using H\"{o}lder's inequality, we obtain that for all $q>1$
\begin{equation*}
\mathbb E \sup_{|\theta|\leq c}| H_{\ell}(\theta,X_T^{m^\ell},X_{T}^{m^{\ell-1}},W_T)| \leq e^{\frac{1}{2} c^{2}T} \left\|e^{ c |W_{T}|} (cT+|W_{T}|) \right\|_{\frac{q}{q-1}} \left\| \left| r_\ell \ (\psi(X_T^{m^\ell})-\psi(X_T^{m^{\ell-1}})) \right|^2 \right\|_q.
\end{equation*}
It is clear that $e^{\frac{1}{2} c^{2}T} \left\|e^{ c |W_{T}|} (cT+|W_{T}|) \right\|_{\frac{q}{q-1}}$ is finite. 
Hence, using \eqref{uniform_integrability} we deduce the boundedness of  $\sup_{\ell} \mathbb E \sup_{|\theta|\leq c}|H_{\ell}(\theta,X_T^{m^\ell},X_{T}^{m^{\ell-1}},W_T)|$. 
According to Lebesgue's theorem, we deduce that $v_{\ell}$ is  $\mathscr C^{1}$ in $\ \mathbb{R}^{q}$ and $ \nabla v_{\ell}(\theta)= \mathbb E H_{\ell}(\theta,X_T^{m^\ell},X_{T}^{m^{\ell-1}},W_T)$. In the same way, we prove that $v_{\ell}$ is of class $\mathscr C^{2}$ in $\mathbb{R}^{q}$. So,  we have 
\begin{multline*}
  \textnormal{Hess}(v_{\ell}(\theta)) =\\
   \mathbb E \left[\left((\theta T-W_{T})(\theta T-W_{T})^{\top}+T I_{q}\right) \left(r_\ell \ (\psi(X_{T}^{m^\ell}) - \psi(X_{T}^{m^{\ell-1}}))\right)^2 e^{- \theta.W_{T} + \frac{1}{2} |\theta|^{2}T}\right].  
\end{multline*}
 Since  $ \mathbb{P}(\psi(X_{T}^{m^\ell})-\psi(X_{T}^{m^{\ell-1}})\neq 0) > 0$,  we get for all $ u \in  \mathbb{R}^{q} \backslash \{0\} $
\begin{multline*}
 u^{\top} \ \textnormal{Hess}(v_{\ell}(\theta)) \ u =\\
 \mathbb E \left[ (T|u|^{2} + (u.(\theta T-W_{T}))^{2}) \left( r_\ell \ (\psi(X_{T}^{m^\ell}) - \psi(X_{T}^{m^{\ell-1}}))\right)^2 \ e^{- \theta.W_{T} + \frac{1}{2} |\theta|^{2}T}\right] > 0.  
\end{multline*}
Hence, $v_{\ell}$ is strictly convex.  Consequently,  to prove that the unique minimum is  attained for a finite value of $\theta$,  it will be sufficient to prove that $ \lim_{|\theta|\rightarrow  +\infty} \ v_{\ell}(\theta) = + \infty$.
Recall that
$  v_{\ell}(\theta)= \mathbb E \left[\left(r_\ell \ (\psi(X_{T}^{m^\ell}) - \psi(X_{T}^{m^{\ell-1}}))\right)^2 e^{- \theta.W_{T} + \frac{1}{2} |\theta|^{2}T}\right].$
Using Fatou's lemma, we get
\begin{multline*}
+\infty=\mathbb E \left[\liminf_{|\theta| \rightarrow +\infty} \left( r_\ell \ (\psi(X_{T}^{m^\ell}) - \psi(X_{T}^{m^{\ell-1}}))\right)^2 e^{- \theta.W_{T} + \frac{1}{2} |\theta|^{2}T}\right] \\\leq \liminf_{|\theta| \rightarrow +\infty} \mathbb E \left[ \left(r_\ell \ (\psi(X_{T}^{m^\ell}) - \psi(X_{T}^{m^{\ell-1}}))\right)^2 e^{- \theta.W_{T} + \frac{1}{2} |\theta|^{2}T} \right].
\end{multline*}
This completes the proof.
\end{proof}
\begin{theorem}
\label{th:convergence_MLMC}
Under the assumptions of propositions \ref{prop:cvx} and \ref{proposition1_MLMC}, we have
$$\theta_{\ell}^{*}{\longrightarrow}\theta^{*},\quad\mbox{ as }{\ell\rightarrow\infty}.$$
\end{theorem}
\begin{proof} 
First of all, we will prove that  $(\theta_{\ell}^{*})_{\ell \in \NN}$ is a $\RR^q$-bounded sequence.  By way of contradiction, let us suppose that there is a subsequence $(\theta_{\ell_k}^{*})_{k\in \NN}$
that diverges to infinity, $\lim_{k\rightarrow\infty}|\theta_{\ell_k}^{*}|=+\infty$. This implies that on the event $\{\psi(X_{T}^{m^\ell}) - \psi(X_{T}^{m^{\ell-1}}) \neq 0\}$, we have the convergence in probability of the quantity
$\bigl( r_\ell \ (\psi(X_{T}^{m^{\ell_k}}) - \psi(X_{T}^{m^{\ell_{k}-1}}))\bigr)^2 e^{- {\theta}^{*}_{\ell_k}W_{T} + \frac{1}{2} |{\theta}^{*}_{\ell_k}|^{2}T}$ towards $+\infty$.
\\ Then, we can extract a subsequence that converges almost surely towards $+\infty$.
Therefore, we apply Fatou's lemma and we get $\lim_{k\rightarrow \infty} v_{\ell_k}(\theta_{\ell_k}^{*})=+\infty$ while 
$$  
v_{\ell_k}(\theta_{\ell_k}^{*})\leq v_{\ell_k}(0) \leq \sup_{\ell} \mathbb E \left( r_\ell ( \psi(X_{T}^{m^{\ell_k}})-\psi(X_{T}^{m^{\ell_{k}-1}})) \right)^{2} <\infty.
$$
The boundedness of the last expression is a consequence of \eqref{uniform_integrability}.
This leads to a contradiction and we deduce that there is some $M>0$ such that $|\theta_{\ell}^{*}|\leq M$ for all $ \ell \in \NN$. 

Now, it remains to prove 
that the set $S=\{x\in \RR^q\, :\,\theta_{\ell_k}^{*}\rightarrow x \mbox{ for some subsequence } \theta_{\ell_k}^{*}\}$ is reduced to the singleton set $\{\theta^{*}\}$. Let us consider a subsequence $\theta_{\ell_k}^{*}\rightarrow \theta_{\infty}^{*}\in S$ as $k$ tends to infinity. According to Proposition \ref{proposition1_MLMC} above, we have 
$$
\nabla v_{\ell_k}(\theta_{\ell_k}^{*})=\mathbb E \left[ (\theta_{\ell_k}^{*} T-W_{T}) 
\bigl( r_\ell \ (\psi(X_{T}^{m^{\ell_k}}) - \psi(X_{T}^{m^{\ell_{k}-1}})) \bigr)^{2} e^{- \theta_{\ell_k}^{*}.W_{T} + \frac{1}{2}|\theta_{\ell_k}^{*}|^{2}T}
\right]=0.
$$

Now, let $q>1$, using Cauchy Schwarz inequality combined with the boundedness of $\theta_{\ell_k}^{*}$ established in the first part of the proof, we check easily  that there exists $c>0$ depending on 
$q$, $T$ and $M$  such that
\begin{multline*}
 \mathbb E \left[ \bigl|(\theta_{\ell_k}^{*} T-W_{T}) 
  \bigl( r_\ell \ (\psi(X_{T}^{m^{\ell_k}}) - \psi(X_{T}^{m^{\ell_{k}-1}}))\bigr)^2 e^{- \theta_{\ell_k}^{*}.W_{T} + \frac{1}{2}|\theta_{\ell_k}^{*}|^{2}T}\bigr|^{q}\right]\leq \\
 c \, \EE^{1/2} \left( r_\ell \ (\psi(X_T^{m^{\ell_k}})-\psi(X_T^{m^{\ell_{k}-1}}) \right)^{4q}.
\end{multline*}
Then, we get thanks to \eqref{uniform_integrability} the uniform integrability which combined with \eqref{cv_stable_MLMC} yields
$$\nabla v(\theta_{\infty}^{*})= \tilde{\mathbb E}\left[ (\theta_{\infty}^{*} T-W_{T}) 
\bigl(\nabla \psi(X_T). U_T \bigr)^2 e^{- \theta_{\infty}^{*}.W_{T} + \frac{1}{2}|\theta_{\infty}^{*}|^{2}T}
\right]= 0.$$
We complete the proof using the uniqueness of the minimum $\theta^*$.
\end{proof}

\section{Stochastic algorithm with projection}\label{algo:sto}

It is clear that once again from a practical point of view the optimal parameter $\theta^*_\ell$ is still not explicit. 
So, the aim of this section is to introduce and study a truncated version of the Robbins-Monro stochastic algorithm approximating $\theta^*_\ell$ 
and obviously $\theta^*$.
To do so, 
let $\mathcal {\tilde B}=(\tilde \Omega,\mathcal {\tilde F},\tilde \PP)$ be the extended probability space endowed with the filtration
$ (\tilde{\mathcal{F}}_{T,i})_{i\geq 0}=(\sigma(W_{t,j},\tilde W_{t,j}, j \leq i, t \leq T))_{i\geq 0}$ where $((W_{t,j})_{t\geq 0},(\tilde W_{t,j})_{t\geq 0},j\geq 1)$ are independent copies of $((W_{t})_{t\geq 0},(\tilde W_{t})_{t\geq 0})$. Here, we recall that $(W_{t})_{t\geq 0}$ is our initial  $q$-dimensional  standard Brownian motion and $(\tilde W_{t})_{t\geq 0}$ is the new independent $q^2$-dimensional standard Brownian introduced in the expression of the limit process $(U_t)_{t\geq 0}$ given  in relation \eqref{eq_U_MLMC}. We denote by  $((U_{t,j})_{t\geq 0},j\geq 1)$ a sequence 
of independent copies of $(U_{t})_{t\geq 0}$.
Let also $\mathcal K \subset \RR^q$ be a compact convex subset satisfying $0\in \overset{\circ}{\mathcal K}$ (the interior of $\mathcal K$). For a deterministic $\theta_0 \in \mathcal K$, we introduce the sequences $(\theta_i)_{i \in \NN}$ and $(\theta_i^{m^\ell})_{i\in\NN}$, $\ell\in\NN$, defined recursively by
\begin{equation}
\label{SAP_MLMC} 
\left\{
\begin{array}{ll} 
\theta_{i+1} &= \Pi_{\mathcal K }\left[ \theta_i - \gamma_{i+1} H(\theta_i, X_{T,i+1}, U_{T,i+1}, W_{T,i+1}) \right], \\
\theta_{i+1}^{m^\ell} &= \Pi_{\mathcal K} \left[ \theta_i^{m^\ell} - \gamma_{i+1} H_\ell(\theta_i^{m^\ell}, X_{T,i+1}^{m^\ell}, X_{T,i+1}^{m^{\ell-1}}, W_{T,i+1}) \right],\quad \theta_{0}^{m^\ell}=\theta_0,
 \end{array}
\right.
\end{equation}
where $\Pi_{\mathcal K}$ is the Euclidean projection onto the constraint set $\mathcal K$, $H$ and $H_{\ell}$ are given respectively by the following expressions
\begin{align*}
  H(\theta,X_T,U_{T},W_T) :&= (\theta T-W_{T}) \left(\nabla \psi(X_{T}) \cdot U_T \right)^2 e^{- \theta\cdot W_{T} + \frac{1}{2} |\theta|^{2}T} \\
 H_{\ell}(\theta,X_T^{m^\ell},X_{T}^{m^{\ell-1}},W_T) :&= (\theta T-W_{T}) \left( r_\ell \ (\psi(X_{T}^{m^\ell}) - \psi(X_{T}^{m^{\ell-1}}) \right)^2 e^{- \theta\cdot W_{T} + \frac{1}{2} |\theta|^{2}T}, \mbox{ for } \ell\geq 1
\end{align*}
and the gain sequence $(\gamma_i)_{i \in \NN}$ is a decreasing sequence of positive real numbers satisfying
\begin{equation}\label{gain_sequence_MLMC}
\sum_{i=1}^{\infty} \gamma_{i}=\infty \quad \mbox{and} \quad \sum_{i=1}^{\infty} \gamma_{i}^{2} < \infty.
\end{equation}
Here, we recall that for $\ell=0$  we take
$$
H_{0}(\theta,X_T^{m^0},X_{T}^{m^{-1}},W_T):=(\theta T-W_{T}) \psi(X_{T}^{m^0})^2e^{- \theta\cdot W_{T} + \frac{1}{2} |\theta|^{2}T}.
$$
\begin{theorem}
\label{constrained_algo_MLMC}
Under the assumptions of propositions \ref{prop:cvx} and \ref{proposition1_MLMC}, the following assertions hold.
\begin{itemize}
 \item [$\bullet$] If $\theta^*=\arg\min v(\theta)$ is unique s.t $\nabla v(\theta^*)=0$ and $\theta^* \in \mathcal K$ then $\theta_i \overset{a.s} {\underset{i \rightarrow\infty}{\longrightarrow}} \theta^*$.
 \item [$\bullet$] If $\theta_{\ell}^*=\arg\min v_{\ell}(\theta)$ is unique s.t $\nabla v_{\ell}(\theta_{\ell}^*)=0$ and $\theta_{\ell}^* \in \mathcal K$ then $\theta^{m^\ell}_i \overset{a.s}{\underset{i \rightarrow\infty}{\longrightarrow}} \theta_{\ell}^*, \;\forall \ell\geq 0$.
\end{itemize}
\end{theorem}
\begin{proof}
Concerning the first assertion, according to Theorem A.1. in Laruelle, Lehalle and Pag\`es \cite{LaruellepagesLehalle} on
Robbins Monro algorithm with projection: to prove that $\theta_i \underset{i \rightarrow\infty}{\longrightarrow} \theta^*$, we need to check firstly that
$$\forall \theta \neq \theta^*, \quad \langle \nabla v(\theta), \theta-\theta^* \rangle >0.$$
This is satisfied using $\nabla v (\theta^*)=0$ and thanks to the convexity of $v$.
Secondly, we have to check the non explosion condition given by
$$\exists C>0 \mbox{ such that } \forall \theta \in \mathcal K, \quad \tilde{\EE} \left[ |H(\theta, X_T, U_T, W_T)|^{2} \right] < C (1+|\theta^2|).$$
Using H\"older's inequality, we can check that for all $1<a'\leq a/2$ we have
\begin{multline}
\tilde{\EE} \left[ |H(\theta, X_T, U_T, W_T)|^{2} \right] \leq e^{ |\theta|^{2}T}
 \EE^{1/a'} \left[|\nabla \psi(X_T)|^{4a'}\right] \\ \tilde{\EE}^{\frac{a'-1}{2a'}} \left[|U_T|^{\frac{8a'}{a'-1}} \right] \EE^{\frac{a'-1}{2a'}} \left[ \left| e^{-\theta \cdot W_{T}} (\theta T-W_{T}) \right|^{\frac{4a'}{a'-1}}\right].
\end{multline}
As $a'\leq a/2$, we deduce the finiteness of  $\mathbb{E}\left[|\nabla{\psi}(X_{T})|^{4a'}\right]$ thanks to our condition 
${({{\bf {\mathcal R}_{\psi,a}}})}$, while the finiteness of  $\tilde{\EE} \left[|U_T|^{\frac{4a'}{a-1}} \right]$ is ensured by Theorem 2.1 in \cite{BenalayaHajjiKebaier}. Consequently, we deduce the existence of $C>0$ such that
\begin{equation}
\tilde{\EE} \left[ |H(\theta, X_T, U_T, W_T)|^{2} \right] \leq C e^{|\theta|^{2}T} \EE^{\frac{a'-1}{2a'}} \left[ \left| e^{-\theta \cdot W_{T}} (\theta T-W_{T}) \right|^{\frac{4a'}{a'-1}}\right]. 
\end{equation}
Therefore, as $\theta \in \mathcal K$, we conclude that $\sup_{\theta \in \mathcal K} \tilde{\EE} \left[ |H(\theta, X_T, U_T, W_T)|^{2} \right] < \infty$.
Concerning the second assertion, we aim to prove that for fixed $\ell\geq 0$, $\theta^{m^\ell}_i \overset{a.s}{\underset{i \rightarrow\infty}{\longrightarrow}} \theta_{\ell}^*$. Here, we only give a proof for $\ell \geq 1$. The case $\ell=0$ is more simple to treat and is based on the same type of arguments as below. So, 
using $\nabla v_{\ell} (\theta_{\ell}^*)=0$ and thanks to the convexity of $v_{\ell}$ ensured by Proposition \ref{proposition1_MLMC}, we prove that 
$$\forall \theta \neq \theta_{\ell}^*, \quad \langle \nabla v_{\ell}(\theta), \theta-\theta_{\ell}^* \rangle >0.$$
For the non explosion condition, we use Cauchy-Schwarz inequality
\begin{multline}
\EE \left[ |H_{\ell}(\theta, X_T^{m^\ell}, X_T^{m^{\ell-1}}, W_T)|^{2} \right] \leq e^{ |\theta|^{2}T}
 \EE^{1/2} \left[\left|r_\ell \ (\psi(X_{T}^{m^\ell}) - \psi(X_{T}^{m^{\ell-1}})\right|^{8}\right] \\ \EE^{\frac{1}{2}} \left[ \left| e^{-\theta \cdot W_{T}} (\theta T-W_{T}) \right|^{{4}}\right].
\end{multline}
Using \eqref{uniform_integrability}, we deduce the existence of $C >0$ such that 
\begin{equation}
\EE \left[ |H_{\ell}(\theta, X_T^{m^\ell}, X_T^{m^{\ell-1}}, W_T)|^{2} \right] \leq C e^{|\theta|^{2}T} \EE^{\frac{1}{2}} \left[ \left| e^{-\theta \cdot W_{T}} (\theta T-W_{T}) \right|^{4}\right]. 
\end{equation}
As $\theta \in \mathcal K$, we conclude that $\sup_{\theta \in  \mathcal K} \EE \left[ |H_{\ell}(\theta, X_T^{m^{\ell}}, X_T^{m^{\ell-1}}, W_T)|^{2} \right] < \infty$.
\end{proof}
The aim of the following result is to give  the asymptotic behavior of $\theta_{i}^{m^\ell}$ when $\ell \rightarrow\infty$.
\begin{theorem}
\label{constrained_rec}
Under the assumptions of propositions \ref{prop:cvx} and \ref{proposition1_MLMC}, we have for all $i\geq 0$ 
$$
\bigl(\theta_{i}^{m^\ell},r_\ell(X_{T,i+1}^{m^\ell}-X_{T,i+1}^{m^{\ell-1}}),W_{T,i+1}\bigr)\overset{\rm stably}{\Longrightarrow}
\bigl(\theta_{i},U_{T,i+1},W_{T,i+1}\bigr) \mbox{ as } \ell\rightarrow\infty,
$$
and consequently 
$$
\theta^{m^\ell}_i \overset{\rm stably}{\Longrightarrow} \theta_i, \mbox{ as } {\ell \rightarrow\infty}.
 $$
\end{theorem}
\begin{proof}
We proceed by induction, the relation is straightforward for $i=0$ as a consequence of \eqref{cv_stable_MLMC}, using the fac that $\theta_0$ is deterministic. 
Next, we assume the induction property holds at the $i^{\rm th}$ step 
and let us prove it at the $(i+1)^{\rm th}$ step. So, we write 
$$
\theta_{i+1}^{m^\ell} = \Pi_{\mathcal K} \left[ \theta_i^{m^\ell} - \gamma_{i+1} 
(\theta_i^{m^\ell}T-W_{T,i+1}) \left( r_\ell (\psi(X_{T,i+1}^{m^\ell}) - \psi(X_{T,i+1}^{m^{\ell-1}})) \right)^2 
e^{- \theta_i^{m^\ell}\cdot W_{T,i+1} + \frac{1}{2} |\theta_i^{m^\ell}|^{2}T}\right].
$$
where $\mathcal K$ is the compact set used in the projection algorithm. 
Thanks to our induction assumption, when $\ell$ tends to infinity, we have 
\begin{multline}\label{lim_rec}
 \theta_i^{m^\ell} - \gamma_{i+1} 
(\theta_i^{m^\ell}T-W_{T,i+1}) \left( r_\ell (\psi(X_{T,i+1}^{m^\ell}) - \psi(X_{T,i+1}^{m^{\ell-1}})) \right)^2 
e^{- \theta_i^{m^\ell}\cdot W_{T,i+1} + \frac{1}{2} |\theta_i^{m^\ell}|^{2}T} \overset{\rm stably}{\Longrightarrow} \\
\theta_i - \gamma_{i+1}(\theta_i T-W_{T,i+1}) \left(\nabla \psi(X_{T,i+1}) \cdot U_{T,i+1} \right)^2 e^{- \theta_i\cdot W_{T,i+1} + \frac{1}{2} |\theta_i|^{2}T}.
\end{multline}
Now, by Theorem 3 in  \cite{KBA},  the limit process $U$ can be written as 
$
U_t=\frac{1}{\sqrt{2}}\sum_{j,k=1}^q Z_t\int_0^t Y_s^{j,k}d\tilde W_s^{j,k},
$
where the $\RR^{d\times d}$-valued process $(Z_t)_{t\geq 0}$ and the $\RR^{d}$-valued process $(Y_t)_{t\geq 0}$ are both adapted to the natural filtration of the initial Brownian motion $(W_t)_{t\geq 0}$ and are independent from the new Brownian motion $(\tilde W_t)_{t\geq 0}$. Note that  
we have the independence between $\theta_{i}$ and  $(X_{T,i+1},U_{T,i+1}, W_{T,i+1})$. 

So, the distribution of the limit in \eqref{lim_rec} 
conditionally on $\theta_i$ and 
$\sigma(W_{t,i+1},0\leq t\leq T)$ is Gaussian and then 
$$
\tilde{\mathbb P}\left(  \theta_i - \gamma_{i+1}(\theta_i T-W_{T,i+1}) \left(\nabla \psi(X_{T,i+1}) \cdot U_{T,i+1} \right)^2 e^{- \theta_i\cdot W_{T,i+1} + \frac{1}{2} |\theta_i|^{2}T}    \in \partial \mathcal K\right)=0, 
$$
where $\partial \mathcal K$ denotes the border of the compact subset  $\mathcal K$. 
This leads us to deduce the stable convergence in distribution of  $\theta_{i+1}^{m^\ell}$ to $\theta_{i+1}$ as $\ell$ tends to infinity.
We complete the proof using once again \eqref{cv_stable_MLMC} and the independence between $\theta^{m^\ell}_{i+1}$ and  $(X^{m^\ell}_{T,i+2}, X^{m^{\ell-1}}_{T,i+2}, W_{T,i+2})$.
\end{proof}

\begin{remark}\label{original:dl}
It is worthwhile to highlight that when combining the second assertion of Theorem \ref{constrained_algo_MLMC} and  Theorem \ref{th:convergence_MLMC} we get
$$
\lim_{\ell \to \infty }\lim_{i\to \infty } \theta^{m^\ell}_i= \theta^*  \mbox { a.s.}
$$
However, when reverting the order of the indexes in the above convergence we get thanks to Theorem \ref{constrained_rec} and the first assertion of Theorem \ref{constrained_algo_MLMC}
$$
\theta^{m^\ell}_i\underset{ \substack{\ell\to \infty}}{\overset{\rm stably}{\Longrightarrow}}  \theta_i \mbox{ and }  \lim_{i\to \infty }  \theta_i =\theta^* \mbox{ a.s. }
$$
To the best of our knowledge such result does not exist in the literature on stochastic algorithms and it is  crucial to prove the main theorems of the next section.
\end{remark}
\section{Adaptive Euler Multilevel Monte Carlo algorithm}\label{CLTadp_MLMC}
In this section, we prove a central limit theorem for the adaptive Euler MLMC method
which approximates our initial quantity of interest  $\mathbb{E} \psi(X_{T})=\EE\left[\psi(X^{\theta}_T) e^{-\theta\cdot W_T
-\frac{1}{2} |\theta|^{2}T}\right]$ by 
\begin{multline}
\label{algo-MLMC}
Q_n:= \frac{1}{N_0} \sum_{i=1}^{N_0}g(\theta_{i-1}^{m^0}, X_{T,i}^{m^0,\theta_{i-1}^{m^0}},\hat{W}_{T,i})\\ 
+\sum_{\ell=1}^{L} \frac{1}{N_\ell} \sum_{i=1}^{N_\ell}\left(g(\theta_{i-1}^{\ell, m^\ell}, X_{T,i}^{\ell,m^\ell,\theta_{i-1}^{\ell, m^\ell}},W_{T,i}^{\ell})-g(\theta_{i-1}^{\ell,m^{\ell}}, 
X_{T,i}^{\ell,m^{\ell-1},\theta_{i-1}^{\ell, m^{\ell}}},W_{T,i}^{\ell})\right),
\end{multline}
where for all  $ x\in\RR^d \mbox{ and } y\in\RR^q$, $g(\theta,x,y)=\psi(x) e^{-\theta\cdot y-\frac{1}{2} |\theta|^{2}T}$, $L=\frac{ \log n}{ \log m}$.
Here, we recall that the $(L+1)$ empirical means are of course independent. The complete procedure combining the adaptive importance sampling based on Robbins-Monro with projection and the Multilevel Monte Carlo Euler method is presented in  Algorithm \ref{algo_Q}.

\begin{algorithm}
\SetKwData{Left}{left}
\SetKwData{This}{this}
\SetKwData{Up}{up}
\SetKwFunction{Union}{Union}
\SetKwFunction{FindCompress}{FindCompress}
\SetKwInOut{Input}{Input}
\SetKwInOut{Comp}{Compute}
\SetKwInOut{Output}{Output}
\Input{A positive
    sequence $\gamma = ({\gamma_i})_{i\geq 1}$, $(m,L)\in\NN^*\times\NN^*$ and $T>0$;
    \\A point $\theta_0 \in \RR^d \times 
    \overset{\circ}{\mathcal K}$\textcolor{forest}{\tcp*[r]{$\mathcal K$ is a convex, compact subset of $\RR^q$.}}
    The parameter $\alpha$ \textcolor{forest}{\tcp*[r]{the rate of the weak error (see $(\mathcal H_{\varepsilon_n})).$ }}
    The sequence $(a_{\ell})_{0\leq \ell \leq L}$ \textcolor{forest}{\tcp*[r] {  satisfying condition $(\bf \mathcal W)$. }}
    }
\Output{The estimator $Q_n$\textcolor{forest}{\tcp*[r]{{$n=m^L$ is the finest step of discretization.}}}}
\Comp{The sample sizes $N_\ell=\frac{n^{2 \alpha} (m-1) T}{m^\ell a_\ell} \sum_{\ell=1}^{L} a_\ell$ for $\ell\in\{0,\cdots,L\};$}
\BlankLine
\emph{Special treatment for the first empirical mean in $Q_n$}\; 
 Initialization: $\theta=\theta_0$, $Q^1_n=0$;
 
\For{$i\leftarrow 1$ \KwTo $N_0$}{
Generate the triplet $(X^{m^0}_T,X^{m^0,\theta}_T,W_T)$;\\
$Q^1_n \leftarrow Q^1_n+g(\theta,X^{m^0,\theta}_T,W_T)$;\\
$\theta  \leftarrow \theta - \gamma_{i} H_0(\theta,X^{m^0}_T,W_T)$;\\
\If(){$\theta\notin \mathcal K$}
{$\theta  \leftarrow \Pi_{\mathcal K}(\theta)$ \textcolor{forest}{\tcp*[r] { $\Pi_{\mathcal K}$ is the Euclidean projection onto $\mathcal K$. }}}
}
$Q^1_n\leftarrow {Q^1_n}/{N_0};$ 
\BlankLine
\emph{Computation of the $L$ empirical means in the second term of $Q_n$}\; 
 Initialization:  $Q^2_n=0$;\\
\For{$\ell \leftarrow 1$ \KwTo $L$}{
 Initialization: $\theta=\theta_0$, $S_n=0$;
 
\For{$i\leftarrow 1$ \KwTo $N_\ell$}{
Generate the quintuplets $(X^{m^\ell}_T,X^{m^{\ell-1}}_T ,X^{m^\ell,\theta}_T,X^{m^{\ell-1},\theta}_T ,W_T)$;
$S_n \leftarrow S_n+g(\theta,X^{m^\ell,\theta}_T ,W_T)-g(\theta,X^{m^{\ell-1},\theta}_T ,W_T) $;\\
$\theta  \leftarrow \theta - \gamma_{i} H_\ell(\theta,X^{m^\ell}_T, X^{m^\ell-1}_T,W_T)$;\\
\If(){$\theta\notin \mathcal K$}
{$\theta  \leftarrow \Pi_{\mathcal K}(\theta)$;

}
}
$Q^2_n\leftarrow Q^2_n +{S_n}/{N_\ell};$ 
}
\BlankLine
$Q_n\leftarrow Q^1_n + Q^2_n$;
\caption{Adaptive importance sampling Multilevel Monte Carlo Euler method}\label{algo_Q}
\end{algorithm}

\subsection{Asymptotic behavior of the variance of the algorithm}
\begin{theorem} \label{CLTMLMC_adaptatif1}
 Let assumptions of propositions \ref{prop:cvx} and \ref{proposition1_MLMC} hold. If moreover assumption $(\mathcal{H}_{\varepsilon_n})$ is satisfied,
then for the choice of $N_\ell, \ell \in \{ 0,1,...,L\}$ given by (\ref{sample_size_MLMC}), the following convergence holds
\begin{equation*}
\mbox{ for } \alpha \in [1/2,1], \quad \lim\limits_{n \to \infty} n^{2 \alpha} \EE \left[\left| Q_n -\mathbb{E}\psi(X_{T}) \right|^2 \right] = \tilde{\sigma}^2 + C_\psi^2,
\end{equation*}
where $C_\psi$ and $\alpha$ are given by relation $(\mathcal{H}_{\varepsilon_n})$ and $\tilde \sigma^2:= {\Et}\left[ \left[  \nabla \psi(X_{T})\cdot U_{T} \right]^{2} e^{-\theta^{*}\cdot W_{T}+\frac{1}{2}|\theta^{*}|^{2}T}  \right]$. 
\end{theorem}
\begin{proof}
At first, we rewrite the total error  as follows
\begin{equation}\label{decomposition}
 n^{\alpha}\left(Q_n -  \mathbb{E}\psi(X_{T})\right) = Q_{n}^{1} + Q_{n}^{2} + n^{\alpha} (\mathbb{E} \psi(X_{T}^{n}) - \mathbb{E} \psi(X_{T})),
\end{equation}
where $Q_{n}^1$ and $Q_{n}^{2}$ are given by the following expressions
\begin{equation}
\label{Qn1}
Q_{n}^{1} := n^{\alpha}\left( \frac{1}{N_0} \sum_{i=1}^{N_0} g({\hat\theta}_{i-1}^{m^{0}}, \hat{X}_{T,i}^{m^0,{\hat\theta}_{i-1}^{m^{0}}},W_{T,i}) - \mathbb{E}\psi(X_{T}^{1}) \right).
\end{equation}
\begin{multline*}
Q_{n}^{2} := n^{\alpha} \left[\sum_{\ell=1}^{L} \frac{1}{N_\ell} \sum_{i=1}^{N_\ell} \left( g(\theta_{i-1}^{\ell,m^{\ell}}, X_{T,i}^{\ell,m^{\ell},{\theta}_{i-1}^{\ell, m^{\ell}}},W_{T,i}) -g({\theta}_{i-1}^{\ell, m^{\ell}}, X_{T,i}^{\ell,m^{\ell-1},\theta_{i-1}^{\ell, m^{\ell}}},W_{T,i}) \right. \right.\\
- \left. \left. \mathbb{E} \left[ \psi(X_{T}^{m^{\ell}}) -\psi(X_{T}^{m^{\ell-1}}) \right] \right) \right].
\end{multline*}
Since $Q_n^1$ and $Q_n^2$ are independent and centered, we write also 
\begin{equation*}
 n^{2\alpha} \EE \left[ \left| Q_n -  \mathbb{E}\psi(X_{T})\right|^2 \right]= \EE\left[ |Q_{n}^{1}|^2 \right] + \EE \left[ |Q_{n}^{2}|^2\right] + n^{2 \alpha} (\mathbb{E} \psi(X_{T}^{n}) - \mathbb{E} \psi(X_{T}))^2.
\end{equation*}
Using assumption $(\mathcal{H}_{\varepsilon_n})$, the last term of the previous expression converges towards the discretization constant $C_{\psi}^2$ as $n$ goes to infinity.
\\ {{$\bullet\,$\bf {Step 1.}}\,} 
For the first term $\EE |Q_{n}^{1}|^2$, noticing that $\left(\sum_{i=1}^{k} g(\theta_{i-1}^{m^{0}}, X_{T,i}^{m^0,\theta_{i-1}^{m^{0}}},W_{T,i}) - \mathbb{E}\psi(X_{T}^{m^0}), k \geq 1 \right)$ is a martingale with respect to the filtration 
$\tilde{\mathcal{F}}_{T,k}$, then we write
\begin{eqnarray}
 \EE|Q_n^1|^2 &=& \EE \left( \frac{n^{2\alpha}}{N_0^2}  \sum_{i=1}^{N_0} \EE \left[ \left| g(\theta_{i-1}^{m^{0}}, X_{T,i}^{m^0,\theta_{i-1}^{m^{0}}},W_{T,i}) - \mathbb{E}\psi(X_{T}^{m^0}) \right|^2 | \tilde{\mathcal{F}}_{T,i-1} \right] \right) \nonumber \\
 &=&\frac{n^{2\alpha}}{N_0^2} \sum_{i=1}^{N_0} \left( \EE \left( \psi(X_{T}^{m^0})^2 e^{-\theta_{i-1}^{m^0} \cdot W_{T}+ \frac{1}{2}|\theta_{i-1}^{m^0}|^{2}T}\right) - \left [\mathbb{E}\psi(X_{T}^{m^0}) \right]^2 \right) \nonumber.
\end{eqnarray}
Since $X_{T,i}^{m^0} \independent \tilde{\mathcal F}_{T,i-1}$ and $\theta_{i-1}^{m^0}$ is $\tilde{\mathcal F}_{T,i-1}$-measurable and thanks to Girsanov theorem, we obtain the last equality by introducing a new couple of random variables $(X_T^{m^0},W_T)$ independent of $\tilde{\mathcal F}_{T}= \cup_{i\geq 1} \tilde{\mathcal F}_{T,i}$.
As $N_0=\frac{n^{2\alpha}(m-1)T}{a_0} \sum_{\ell=1}^{L} a_\ell$, we write
\begin{equation}\label{first_term}
\EE|Q_n^1|^2 =\frac{ a_0}{(m-1)T \sum_{\ell=1}^L a_\ell} \frac{1}{N_0} \sum_{i=1}^{N_0} \left( \EE \left( \psi(X_{T}^{m^0})^2 e^{-\theta_{i-1}^{m^0} \cdot W_{T}+ \frac{1}{2}|\theta_{i-1}^{m^0}|^{2}T}\right) - \left [\mathbb{E}\psi(X_{T}^{m^0}) \right]^2 \right). 
\end{equation}
Since $(\theta_{i-1}^{m^0})_{i \in \NN} \in \mathcal K$, we have the existence of 
$c >0,$ s.t. $\sup_{i \in \NN} \left|\psi(X_{T}^{m^0})^2 e^{-\theta_{i-1}^{m^0} \cdot W_{T}+ \frac{1}{2}|\theta_{i-1}^{m^0}|^{2}T} \right|
\leq \psi(X_{T}^{m^0})^2 e^{c |W_{T}|+ \frac{c^{2}}{2}T}.
$
This upper bound is clearly integrable using the H\"older's inequality combined with properties  $(\mathcal{P})$ and \eqref{condition_dif_psi1}.
Therefore, by the Lebesgue's theorem and thanks to the  second assertion of Theorem \ref{constrained_algo_MLMC}, we obtain that 
$$ \lim\limits_{i \to \infty} \EE \left( \psi(X_{T}^{m^0})^2 e^{-\theta_{i-1}^{m^0} \cdot W_{T}+ \frac{1}{2}|\theta_{i-1}^{m^0}|^{2}T}\right) =\EE \left( \psi(X_{T}^{m^0})^2 e^{-\theta^{*}_{0} \cdot W_{T}+ \frac{1}{2}|\theta^{*}_{0}|^{2}T}\right).$$
Thus, using Cesaro's lemma and  $\sum_{\ell=1}^{L} a_\ell \underset{L \rightarrow \infty} {\longrightarrow} \infty $, we conclude that
$
\EE|Q_n^1|^2 \underset {n \rightarrow \infty} {\longrightarrow} 0.
$
\vspace{0.5cm}
\\{{$\bullet\,$\bf {Step 2.}}\,} 
For the second term, we write
$
 \EE |Q_n^2|^2 = \EE [ (\sum_{\ell=1}^{L}  \frac{n^\alpha}{N_\ell} \sum_{i=1}^{N_\ell} Z_{T,i}^{m^\ell, m^{\ell-1}, \theta_{i-1}^{\ell,m^\ell}})^2 ], 
$
where
\begin{multline}
\label{Z_CLT_MLMC}
Z_{T,i}^{m^\ell, m^{\ell-1}, \theta_{i-1}^{\ell,m^\ell}}=\left(g(\theta_{i-1}^{\ell,m^\ell}, X_{T,i}^{\ell,m^\ell,\theta_{i-1}^{\ell,m^\ell}},W_{T,i}^\ell)-g(\theta_{i-1}^{\ell,m^\ell}, X_{T,i}^{\ell,m^{\ell-1},\theta_{i-1}^{\ell,m^\ell}},W_{T,i}^\ell) \right. \\
  - \left. \mathbb{E}[\psi(X_{T}^{m^\ell})- \psi(X_{T}^{m^{\ell-1}})]\right).
\end{multline}
As $(Z_{T,i}^{m^\ell, m^{\ell-1}, \theta_{i-1}^{\ell,m^\ell}})_{1 \leq \ell \leq L}$ are independent, we get
$
 \EE |Q_n^2|^2 = \sum_{\ell=1}^{L} \frac{n^{2\alpha}}{N_\ell^2} \EE ( \sum_{i=1}^{N_\ell} Z_{T,i}^{m^\ell, m^{\ell-1}, \theta_{i-1}^{\ell,m^\ell} } )^2. 
$
By the same kind of arguments as in the first step, noticing that for each $\ell \in \{1,...,L\}$, we have  $(\sum_{i=1}^{k} Z_{T,i}^{m^\ell, m^{\ell-1}, \theta_{i-1}^{\ell,m^\ell}}, k \geq 1)$ is $\tilde{\mathcal F}_{T,k}$ martingale, we write 
\begin{eqnarray*}
\EE |Q_n^2|^2  &=& \EE \left( \sum_{\ell=1}^{L} \frac{n^{2\alpha}}{N_\ell^2} \sum_{i=1}^{N_\ell} \EE \left[ \left(Z_{T,i}^{m^\ell, m^{\ell-1}, \theta_{i-1}^{ m^\ell}}  \right)^2 | \tilde{\mathcal{F}}_{T,i-1}\right] \right)\\
 &=&\sum_{\ell=1}^{L} \frac{n^{2\alpha}}{N_{\ell}}  \left[ \frac{1}{N_\ell} \sum_{i=1}^{N_\ell} \left(\EE\left( [\psi(X_{T}^{m^\ell})-\psi(X_{T}^{m^{\ell-1}})]^2 e^{-\theta_{i-1}^{m^\ell} \cdot W_{T}+\frac{1}{2}|\theta_{i-1}^{m^\ell}|^{2}T}\right) \right. \right. \\ & &\hspace{6cm}\left. \left.-\left(\EE\psi(X_{T}^{m^\ell})-\EE\psi(X_{T}^{m^{\ell-1}})\right)^2 \right) \right].
\end{eqnarray*}
Since for each $\ell \in \{1,...,L\}$, $X_{T,i}^{\ell,m^\ell} \independent \tilde{\mathcal F}_{T,i-1}$ and $\theta_{i-1}^{m^\ell}$ is $\tilde{\mathcal F}_{T,i-1}$-measurable and thanks to Girsanov theorem, we obtain the last equality by introducing a new couple of random variables $(X_T^{m^\ell},W_T)$ independent of $\tilde{\mathcal F}_{T}= \cup_{i\geq 1} \tilde{\mathcal F}_{T,i}$.
As $N_\ell=\frac{n^{2 \alpha} (m-1) T}{m^\ell a_\ell} \sum_{\ell=1}^{L} a_\ell, \ \ell \in \{0,...,L\}$ (see relation \eqref{sample_size_MLMC}), we write
\begin{equation}\label{crochetA122_MLMC}
  \EE |Q_n^2|^2 = \frac{1}{\sum_{\ell=1}^{L} a_\ell}  \sum_{\ell=1}^{L}  a_\ell \mathcal B_{\ell}   
   - \frac{1}{\sum_{\ell=1}^{L} a_\ell}  \sum_{\ell=1}^{L}  a_\ell \left( r_\ell \ \left[\EE\psi(X_{T}^{m^\ell})-\EE\psi(X_{T}^{m^{\ell-1}}) \right] \right)^2,
\end{equation}
where 
$$ 
\mathcal B_{\ell}:=\frac{1}{N_\ell} \sum_{i=1}^{N_\ell}\EE\left( [r_\ell(\psi(X_{T}^{m^\ell})-\psi(X_{T}^{m^{\ell-1}}))]^2 e^{-\theta_{i-1}^{m^\ell} \cdot W_{T}+\frac{1}{2}|\theta_{i-1}^{m^\ell}|^{2}T}\right).
$$
Now, for the last term of relation (\ref{crochetA122_MLMC}), as assumption $(\mathcal H_{\psi})$  is satisfied under $(\mathcal R_{\psi,a}$), then relation \eqref{result1_MLMC} remains valid. 
Therefore, $ r_\ell \left(\EE (\psi(X_{T}^{m^\ell}))-\EE( \psi(X_{T}^{m^{\ell-1}})) \right) $ converges to $\tilde{\EE} \left( \nabla \psi(X_T).U_T \right)=0$ as $\ell$ goes to infinity an by Toeplitz lemma, we obtain that
\begin{equation*}
\lim\limits_{\ell \rightarrow \infty} \frac{1}{\sum_{\ell=1}^{L} a_\ell}  \sum_{\ell=1}^{L}  a_\ell \left( r_\ell \left[\EE\psi(X_{T}^{m^\ell})-\EE\psi(X_{T}^{m^{\ell-1}}) \right] \right)^2 =0.
\end{equation*}
Hence, to complete the proof it is sufficient now to prove 
$
\lim_{\ell\rightarrow\infty} \mathcal B_\ell = \tilde\sigma^2.
$
\vspace{0.5cm}
\\{{$\bullet\,$\bf {Step 3.}}\,} 
To do so, we intend to use Lemma \ref{Toep:ex}, in the appendix section. More precisely, we have to prove that
$$
\lim_{\ell\rightarrow\infty}\lim_{i\rightarrow\infty} b_{i,\ell}=\lim_{i\rightarrow\infty}\lim_{\ell\rightarrow\infty}b_{i,\ell}=\tilde\sigma^2,
$$
where 
$$
b_{i,\ell}:=\EE\left( [r_\ell(\psi(X_{T}^{m^\ell})-\psi(X_{T}^{m^{\ell-1}}))]^2 e^{-\theta_{i-1}^{m^\ell} \cdot W_{T}+\frac{1}{2}|\theta_{i-1}^{m^\ell}|^{2}T}\right).
$$
For the first double limit, since $\theta_{i-1}^{m^\ell} \in \mathcal K$, we have the existence of a positive constant $c$ s.t.
\begin{equation*}
\sup_{i \in \NN} \left|\left(\psi(X_{T}^{m^\ell})-\psi(X_{T}^{m^{\ell-1}})\right)^2 e^{-\theta_{i-1}^{m^\ell} \cdot W_{T}+\frac{1}{2}|\theta_{i-1}^{m^\ell}|^{2}T} \right|
\leq \left(\psi(X_{T}^{m^\ell})-\psi(X_{T}^{m^{\ell-1}})\right)^2 e^{c|W_{T}|+\frac{c^2}{2}T}.
\end{equation*}
This upper bound is clearly integrable using the Cauchy Schwarz inequality combined with  property \eqref{uniform_integrability}.
Therefore, by the Lebesgue's theorem and the second assertion of Theorem \ref{constrained_algo_MLMC}, we obtain that 
$
\lim_{i \rightarrow \infty} b_{i,\ell}= 
\EE\left( [r_\ell (\psi(X_{T}^{m^\ell})-\psi(X_{T}^{m^{\ell-1}}))]^2 e^{-\theta_{\ell}^{*} \cdot W_{T}+\frac{1}{2}|\theta_{\ell}^{*}|^{2}T}\right). 
$
Now, thanks to \eqref{cv_stable_MLMC} and Theorem \ref{th:convergence_MLMC}, we have 
\begin{equation}\label{stablecv_MLMC}
r_\ell \ [\psi(X_{T}^{m^\ell})-\psi(X_{T}^{m^{\ell-1}})] e^{-\frac{1}{2}\theta_{\ell}^{*}\cdot W_{T}+\frac{1}{4}|\theta_{\ell}^{*}|^{2}T}
 \overset{\rm stably}{\underset{\ell \rightarrow \infty}{\Longrightarrow }} \nabla \psi(X_{T})\cdot U_{T} e^{-\frac{1}{2}\theta^{*}\cdot W_{T}+\frac{1}{4} | \theta^{*}|^{2}T}.
\end{equation}
Moreover, for $q>1$ we have by Cauchy-Schwartz inequality
\begin{multline}\label{Uniform_integrability_first}
 \mathbb{E} \left| r_\ell \ [\psi(X_{T}^{m^\ell})-\psi(X_{T}^{m^{\ell-1}})] e^{-\frac{1}{2} \theta_{\ell}^{*}\cdot W_{T}+\frac{1}{4}|\theta_{\ell}^{*}|^{2}T}\right|^{2 q} \\
 \leq r_\ell^{2q} \left[\mathbb{E}  \left| \psi(X_{T}^{m^\ell})-\psi(X_{T}^{m^{\ell-1}}) \right|^{4q}
 \right]^{\frac{1}{2}}e^{ \frac{q(2q+1)}{2} |\theta_{\ell}^{*}|^{2} T}.
\end{multline}
Hence,  as  $(\theta_{\ell}^{*})_{ \ell\ge 1 } \in \mathcal K$,  we obtain the uniform integrability of the sequence $ ( (r_\ell [\psi(X_{T}^{m^\ell})-\psi(X_{T}^{m^{\ell-1}})])^2
 e^{-\theta_{\ell}^{*} \cdot W_{T}+\frac{1}{2}|\theta_{\ell}^{*}|^{2}T})_{\ell\geq 1}$ thanks to property  \eqref{uniform_integrability}. 
Then, by the stable convergence obtained in (\ref{stablecv_MLMC}), we deduce that
$ \lim_{\ell\rightarrow\infty}\lim_{i\rightarrow\infty} b_{i,\ell}=\tilde\sigma^2$.

Now, concerning the second double limit, by Theorem \ref{constrained_rec}, we have for all $i\geq 1$ 
$$
\bigl(\theta_{i-1}^{m^\ell},r_\ell(X_{T}^{m^\ell}-X_{T}^{m^{\ell-1}}),W_{T}\bigr)\overset{\rm stably}{\Longrightarrow}
\bigl(\theta_{i-1},U_{T},W_{T}\bigr) \mbox{ as } \ell\rightarrow\infty.
$$
Since for all $i\in\NN$,  $(\theta_{i-1}^{m^\ell})_{ 1\leq \ell \leq L}  \in \mathcal K$, we use the same type of arguments as above,  to get the uniform integrability
and then deduce that $\lim_{\ell\rightarrow\infty}b_{i,\ell}=\tilde{\mathbb E}\left[ (\nabla \psi(X_{T})\cdot U_{T})^2 e^{-\theta_{i-1}\cdot W_{T}+\frac{1}{2} | \theta_{i-1}|^{2}T}\right].$ We complete the proof by applying the Lebesgue's theorem together with the first assertion of Theorem \ref{constrained_algo_MLMC}.
\end{proof}
Our aim now is to prove a central limit theorem for the adaptive Euler MLMC method.
\subsection{Central Limit Theorem}
\begin{theorem}\label{CLTMLMC_adaptatif2}
 Under assumptions of Theorem \ref{CLTMLMC_adaptatif1} and for the choice of $N_\ell, \ell \in \{ 0,1,...,L\}$ given by (\ref{sample_size_MLMC}), the following convergence holds
\begin{equation*}
n^{\alpha}\left(Q_n -
 \mathbb{E}\psi(X_{T})\right) \underset{n \rightarrow \infty} {\overset{\mathcal{L}}{\longrightarrow}} \mathcal{N} \left(C_{\psi}, \tilde{\sigma}^2 \right),
\end{equation*}
where
$\tilde \sigma^2:= {\Et}\left[ \left[  \nabla \psi(X_{T})\cdot U_{T} \right]^{2} e^{-\theta^{*}\cdot W_{T}+\frac{1}{2}|\theta^{*}|^{2}T}  \right]$.
\end{theorem}
\begin{proof}
We consider the same decomposition given by relation \eqref{decomposition} in the beginning of the proof of Theorem \ref{CLTMLMC_adaptatif1}. 
Under assumption $(\mathcal{H}_{\varepsilon_n})$, the last term on the right hand side of this relation converges to $C_\psi$ as $n$ goes to $\infty$.
For the convergence of the term $Q_n^1$, we use the result of the first step in the proof of Theorem \ref{CLTMLMC_adaptatif1} to deduce that $Q_n^1 \overset{\PP}{\underset{n \rightarrow \infty} \longrightarrow} 0$. 
Concerning the convergence of the term $Q_n^2$, we plan to use the Lindeberg-Feller central limit theorem (see Theorem \ref{CLT Lindeberg Feller}) with the Lyapunov condition.
We introduce the independent random variables 
\begin{equation}
\label{SNL}
(S_{\ell,N_\ell}, \ell \in \{1,...,L\}) \mbox{ where } (S_{\ell,k}= \frac{n^\alpha}{N_\ell} \sum_{i=1}^{k} Z_{T,i}^{m^\ell, m^{\ell-1}, \theta_{i-1}^{\ell,m^\ell}}, k \geq 1) \mbox{ for } \ell \in \{ 1,...,L\},
\end{equation}
and we need to check the assertions A1. and A3. in Theorem \ref{CLT Lindeberg Feller}. 
More precisely, we will prove
\begin{itemize}
 \item [A1.] $\lim\limits_{L \rightarrow \infty} \sum_{\ell=1}^{L} \EE (S_{\ell,N_\ell}^2) = \tilde {\mathbb{E}} \left(\left[  \nabla \psi(X_{T})\cdot U_{T} \right]^{2} e^{-\theta^* \cdot W_{T}+\frac{1}{2}|\theta^*|^{2}T}  \right).$ 
 \item [A3.] For $p>2$, $\lim\limits_{L \rightarrow \infty} \sum_{\ell=1}^{L} \EE (S_{\ell,N_\ell}^p)=0$.
\end{itemize}
As $ \sum_{\ell=1}^{L} \EE (S_{\ell,N_\ell}^2)= \EE(Q_n^2)^2$,  assertion A1. is nothing but the result of the second step in the proof of Theorem \ref{CLTMLMC_adaptatif1}. 
Now, it remains to verify the assertion A3. We get by Burkholder's inequality (see Theorem 2.10 in \cite{HPHCC}): for $p>2$, there exists $C_p >0$ such that
\begin{equation}
\label{S_lyapunov_MLMC}
 \EE \left| S_{\ell,N_\ell} \right|^p \leq C_p \frac{n^{\alpha p}}{N_\ell^{p/2}} \EE \left[ \frac{1}{N_\ell} \sum_{i=1}^{N_\ell} (Z_{T,i}^{m^\ell, m^{\ell-1}, \theta_{i-1}^{\ell, m^\ell}})^2 \right]^{p/2}.
\end{equation}
Moreover, by using Jensen's inequality, we obtain for $p>2$
\begin{equation*}
 \sum_{\ell=1}^{L} \EE \left| S_{\ell, N_\ell} \right|^p \leq \sum_{\ell=1}^{L} C_p \frac{n^{\alpha p}}{N_\ell^{p/2}} \EE \left[ \frac{1}{N_\ell} \sum_{i=1}^{N_\ell} |Z_{T,i}^{m^\ell, m^{\ell-1}, \theta_{i-1}^{\ell,m^\ell}}|^p \right].
 \end{equation*}
Using that $N_\ell=\frac{n^{2 \alpha} (m-1) T}{m^\ell a_\ell} \sum_{\ell=1}^{L} a_\ell$ (see relation \eqref{sample_size_MLMC}) and by conditioning, we get 
\begin{equation*}
\sum_{\ell=1}^{L} \EE \left| S_{\ell,N_\ell} \right|^p \leq \sum_{\ell=1}^{L} C_p \frac{a_\ell^{p/2}}{(\sum_{\ell=1}^L a_\ell)^{p/2}} \EE \left( \frac{1}{N_\ell} \sum_{i=1}^{N_\ell} \EE\left[ r_\ell^p |Z_{T,i}^{m^\ell, m^{\ell-1}, \theta_{i-1}^{\ell, m^\ell}}|^p  | \tilde{\mathcal{F}}_{T,i-1}\right] \right)
\end{equation*}
We write also
\begin{equation}
\label{lyapunov_relation}
 \sum_{\ell=1}^{L} \EE \left| S_{\ell,N_\ell} \right|^p \leq 2^{p-1} C_p \sum_{\ell=1}^{L} \frac{a_\ell^{p/2}}{(\sum_{\ell=1}^L a_\ell)^{p/2}} A_\ell + 2^{p-1} C_p \sum_{\ell=1}^{L} \frac{a_\ell^{p/2}}{(\sum_{\ell=1}^L a_\ell)^{p/2}} B_\ell.
\end{equation}
where
 \begin{eqnarray*}
 A_\ell &=&  \frac{1}{N_\ell} \sum_{i=1}^{N_\ell} b_{i,\ell,p}, \mbox{ with }  b_{i,\ell,p}=\EE \left[ \left| r_\ell \ (\psi(X_T^{m^\ell}) - \psi(X_T^{m^{\ell-1}})) \right|^p e^{-(p-1) \theta_{i-1}^{m^\ell} \cdot W_T -(p/2- 3/2) |\theta_{i-1}^{m^\ell}|^2 T}\right],\\
 B_\ell &=& \left|\EE \left[ r_\ell \ (\psi(X_T^{m^\ell}) - \psi(X_T^{m^{\ell-1}})) \right] \right|^p,
 \end{eqnarray*}
since for each $\ell \in \{1,...,L\}$, $X_{T,i}^{\ell,m^\ell} \independent \tilde{\mathcal F}_{T,i-1}$ and $\theta_{i-1}^{m^\ell}$ is $\tilde{\mathcal F}_{T,i-1}$-measurable and thanks to Girsanov theorem, we obtain the last equality \eqref{lyapunov_relation} by introducing a new couple of random variables $(X_T^{m^\ell},W_T)$ independent of $\tilde{\mathcal F}_{T}= \cup_{i\geq 1} \tilde{\mathcal F}_{T,i}$. Using relation \eqref{result1_MLMC} together with assumption $(\mathcal{W})$ (i.e. $\lim_{\ell \rightarrow \infty} \frac{1}{(\sum_{\ell=1}^{L} a_\ell)^{p/2}} \sum_{\ell=1}^{L} a_\ell^{p/2}=0$) and Toeplitz lemma, we deduce that 
\begin{equation}\label{eq:Bl}
 \lim\limits_{L \rightarrow \infty} \frac{1}{\sum_{\ell=1}^L a_\ell^{p/2}}\sum_{\ell=1}^{L} a_\ell^{p/2} B_\ell =0.
\end{equation}
Concerning the first term on the right hand side in \eqref{lyapunov_relation}, we follow the same arguments as in third step of the proof of Theorem 
\ref{CLTMLMC_adaptatif1} to get
$$
\lim_{\ell\rightarrow\infty}\lim_{i\rightarrow\infty} b_{i,\ell,p}=\lim_{i\rightarrow\infty}\lim_{\ell\rightarrow\infty}b_{i,\ell,p}=
\tilde {\mathbb{E}} \left(\left|  \nabla \psi(X_{T})\cdot U_{T} \right|^{p} e^{-\theta^* \cdot W_{T}+\frac{1}{2}|\theta^*|^{2}T}  \right).
$$
Then, by applying Lemma \ref{Toep:ex}, we obtain
$
\lim\limits_{\ell \rightarrow \infty} A_\ell=\tilde {\mathbb{E}} \left(\left|  \nabla \psi(X_{T})\cdot U_{T} \right|^{p} e^{-\theta^* \cdot W_{T}+\frac{1}{2}|\theta^*|^{2}T}  \right). 
$
Finally, by using once again assumption $(\mathcal{W})$, we conclude that 
\begin{equation}\label{eq:Al}
 \lim\limits_{L \rightarrow \infty} \frac{1}{\sum_{\ell=1}^L a_\ell^{p/2}}\sum_{\ell=1}^{L} a_\ell^{p/2} A_\ell =0.
\end{equation}
We complete the proof using \eqref{lyapunov_relation}, \eqref{eq:Bl} and  \eqref{eq:Al}.
\end{proof}
\section{Numerical illustration}\label{hestonmodel_MLMC}
Let us recall that in our setting, $n$ stands for the finest number of steps, $L$ for the 
final level and both satisfy $n=m^L$, where $m\in \mathbb N\setminus\{0,1\}$ is the number of refiners to be fixed later on.  In this context, the total error refers to the error induced by the  discretization scheme and the MLMC approximation. Hence, the discretization error mainly depends on the choice of the approximation scheme we use. Therefore, to achieve a total error of order $\varepsilon\in(0,1)$ and if we use a discretization scheme satisfying   $(\mathcal H_{\varepsilon_n})$ with $\alpha \in [1/2,1]$ one has to choose the finest number of steps $n$ to be proportional to $\varepsilon^{-\alpha}$ and choose the MLMC sample sizes  $(N_\ell)_{\ell\geq 0}$ satisfying assumptions of Theorem \ref{CLTMLMC_adaptatif2}, so that the total error $\varepsilon$ is of order $1/n^{\alpha}$.
This choice, leads to a  time complexity for the standard  MLMC  approach given by 
\begin{equation*}
\begin{split}
C_{\rm MLMC} &= C \times \left(N_0 + \sum_{\ell=1}^{L} N_\ell (m^\ell + m^{\ell -1})\right) \quad \mbox{with  } C>0 \\
&= C \times \left(\frac{n^{2 \alpha} (m-1) T}{a_0} \sum_{\ell=1}^{L} a_\ell + n^{2 \alpha} \frac{(m^2 -1) T}{m} \sum_{\ell=1}^{L} \frac{1}{a_\ell} \sum_{\ell=1}^{L} a_\ell \right) \quad \mbox{with  } C>0.
\end{split}
\end{equation*}
This time complexity reaches its minimum for the specific choice $a_\ell^{*}=1, \ \ell \in \{ 1,...,L\}$.
Hence, the optimal time complexity of the standard MLMC method is given by
\begin{equation*}
C^*_{\rm MLMC} = C \times \left(\frac{(m-1) T}{a_0 \log m } n^{2 \alpha} \log n + \frac{(m^2 -1) T}{m (\log m)^2} n^{2 \alpha} (\log n)^2  \right) \propto n^{2 \alpha} (\log n)^2.
\end{equation*}
This is of course in line with the original parametrization given in \cite{Giles}. 
\\

Due to its design, it is clear that the adaptive importance sampling  Multilevel Monte Carlo approach (AIS MLMC) is more time consuming than the standard MLMC method (see Algorithm \ref{algo_Q}). 
However, in practice we do not need to reach the optimal variance but just to be close enough to it.  
Based on this idea,  we enforce the adaptive stochastic algorithm to stop after $I\in\mathbb N$ iterations. 
Therefore, the time complexity of the stopped AIS MLMC method is given by 
\begin{align*}
 C_{\rm AIS\,MLMC} &= C \times I\times \sum_{\ell=0}^L m^{\ell} + C' \times \left(\sum_{\ell=0}^{L} N_\ell (m^\ell + m^{\ell -1})\right)\quad \mbox{with  } C,\,C'>0.
\end{align*}
For the same specific choice $a_\ell^{*}=1$,  the  optimal complexity is then given by 
\begin{align*}
 C_{\rm AIS\,MLMC} \propto n^{2 \alpha} (\log n)^2\left(1+ \frac{I}{n^{\alpha}(\log n)^2} \right). 
\end{align*}
According to our numerical simulations (see Figure \ref{fig:1}), the performance of the AIS MLMC method seems to be quite similar if it is stopped after $I=15000$ or $I=1000$ iterations. Hence, it is comforting to notice that our AIS MLMC method approximates efficiently the optimal parameter $\theta^*$ reducing the total variance after just $1000$ iterations. 

For our numerical tests, we consider the problem of option pricing under the Black \& Scholes model.
So, we consider a one-dimensional asset price $(S_t)_{0\leq t, \leq T}$, $T>0$  as solution to   
$$
dS_t=rS_t+\sigma S_t dW_t, \quad S_0=s_0>0,
$$
where $(W_t)_{0\leq t, \leq T}$ is a standard Brownian motion, $r>0$ is the interest rate
and $\sigma$ is the constant volatility of the model. 
We propose, to compute the price of an European Call option with a strike $K$ and maturity $T$ given by $e^{-rT}\mathbb E(S_T-K)_+$. Knowing the benchmark  value of such option, we compare standard MLMC method with our AIS MLMC algorithm with $I=15000$ and $I=1000$ for different values of the finest step $n=m^L$. For each method and each value $n$, we repeat this procedure $M=50$ times and we compute the corresponding CPU time and the root mean squared-error given by
\begin{equation}
RMSE= \sqrt{\frac{1}{M} \sum_{i=1}^{M} (\mbox{Simulated value} -\mbox{Benchmark value} )^2.} 
\label{MSE_MLMC}
\end{equation}
Hence,  for $K=100$, $s_0=130$, $T=1$, $r=\log(1.1)$, $\sigma=0.6$ and $m=4$ we provide a couple of points (RMSE, CPU time) which are plotted on Figure \ref{fig:1}. For these values the benchmark  value  is
$49.898585$. In order to have all curves plotted on the same range, we choose $n\in\{125,175,\dots,925\}$ for the MLMC method and $n\in\{100,150,\dots,550\}$ for the AIS MLMC methods. 

Concerning the stochastic approximation, we take the gain sequence given by $\gamma_i=1/(i+1)$ for $i\in\{0,\cdots,I\}$ and we implement an online adaptive version of the popular Rupert \& Poliak method  known for stabilizing numerically the convergence of the Robbins-Monro type algorithms (see e.g. \cite{PM}).
More precisely,  for  $i\in\{0,\cdots,I\}$ we plug  in \eqref{algo-MLMC}, the approximation  
$
\tilde \theta^{m^{\ell}}_{i+1}=\frac{1}{i+1}\sum_{k=0}^i \theta^{m^{\ell}}_k
$
instead of $\theta^{m^{\ell}}_{i+1}$ given by relation \eqref{SAP_MLMC}. For the Euclidean projection, we take the compact set $\mathcal K=[-10,10]$.
We also test the Chen's projection algorithm (see e.g. \cite{LapLel}), though we did not consider this procedure for our adaptive version of the MLMC method.  More precisely, in the setting of the Black \& Scholes model, we consider  an increasing sequence of compact sets
$(\mathcal{K}_{i})_{i\in \NN}$  containing $0$, satisfying
$\cup_{i=0}^{\infty}\; \mathcal{K}_{i} = \mathbb{R}$  and   $\mathcal{K}_{i} \subsetneq \overset{\circ}{\mathcal{K}}_{i+1}, \forall i \in \NN$,
with $\theta_{0}\in \mathcal{K}_{0} $. For a  gain sequence
$(\gamma_i)_{i\in\NN}$ satisfying (\ref{gain_sequence_MLMC}) and $\ell\geq 0$, we set $(\theta^{m^{\ell}}_0,\alpha_{0}^{m^{\ell}})=(\theta_0,0)$ and define recursively the sequence $(\theta_{i}^{m^{\ell}},\alpha_{i}^{m^{\ell}} )_{i\in\NN}$  by
\begin{equation}\label{chenalgo1_MLMC}
\left\{
\begin{array}{ll}
{\mbox{\bf  if} }&  \theta_i^{m^\ell} - \gamma_{i+1} H_\ell(\theta_i^{m^\ell}, X_{T,i+1}^{m^\ell}, X_{T,i+1}^{m^{\ell-1}}, W_{T,i+1})  \in \mathcal{K}_{\alpha_{i}^{m^{\ell}}},\; {\mbox{\bf     then}}\\
&\theta_{i+1}^{m^\ell} = \theta_i^{m^\ell} - \gamma_{i+1} H_\ell(\theta_i^{m^\ell}, X_{T,i+1}^{m^\ell}, X_{T,i+1}^{m^{\ell-1}}, W_{T,i+1}) , \mbox{ and } \alpha_{i+1}^{m^{\ell}}=\alpha_{i}^{m^{\ell}}\\
{\mbox{\bf  else}}&
\theta_{i+1}^{m^\ell} =\theta_{0}  \mbox{ and } \alpha_{i+1}^{m^{\ell}}=\alpha_{i}^{m^{\ell}},
\end{array}
\right.
\end{equation}
 where $H_\ell$ is introduced in Section \ref{algo:sto} (see relation \eqref{SAP_MLMC} and the 
 definition below).  Of course, like for the stochastic algorithm with compact set projection, we stop the above routine after $I$ iterations.
According to Table \ref{table:time reduction 2dim Yinf1}, the gain factor becomes more important when we consider a smaller RMSE. Actually, for a fixed RMSE of order $10^{-4}$, the AIS MLMC reduces the CPU time by a factor $>2$ in comparison to the MLMC one.
\begin{table}[H]
{\renewcommand{\arraystretch}{1} 
{\setlength{\tabcolsep}{0.6cm}
\begin{center}
\begin{tabular} {*{3}{c}} 
    \hline
      RMSE  & MLMC  CPU time (sec.)  &  AIS MLMC with $I=1000$ CPU time (sec.)   \\
    \hline
      $12\cdot10^{-2}$ & $1455$ & $1076$ \\

     $6\cdot 10^{-2}$ & $6340$ & $4152$ \\

     $45\cdot 10^{-3}$ & $16206$ & $7840$ \\

     \hline
\end{tabular} 
 \caption{Time complexity reduction MLMC  versus AIS MLMC.}
 \label{table:time reduction 2dim Yinf1}
\end{center} 
}}
\end{table} 

According to Figure \ref{fig:1}, both stochastic algorithms, improve in a similar way, the performance of the MLMC method.
\begin{figure} [H]
  \begin{center}
      
    \includegraphics [width=0.7 \textwidth] {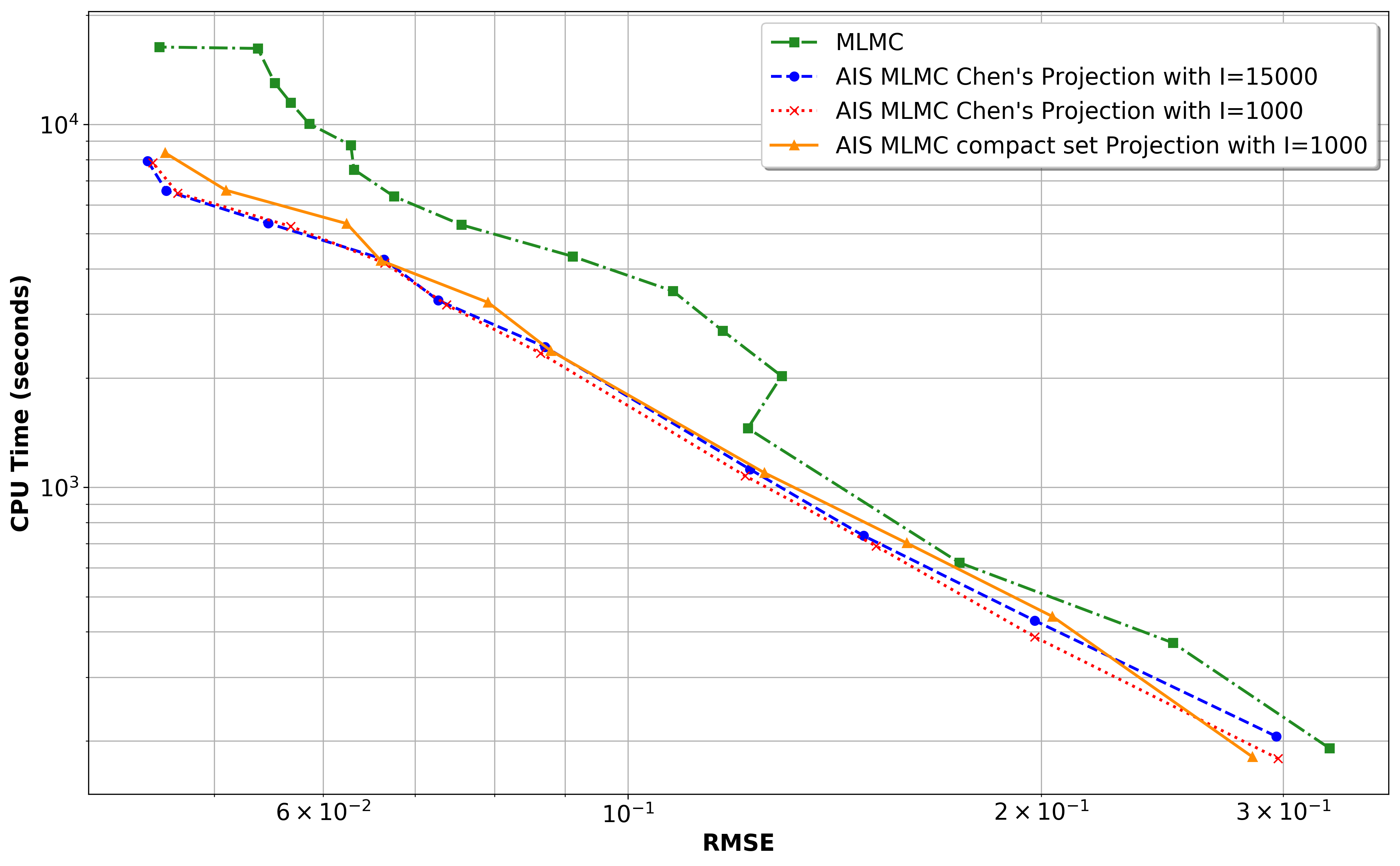}

  \caption{{{\it CPU time vs. RMSE for an European call option under the Black \& Scholes model.}}}
 \label{fig:1}
   \end{center}
\end{figure}   

\section{Appendix}
\subsection{Proof of Lemma \ref{lem:tech}}
The first assertion is an immediate concequence of properties $(\bf \mathcal P)$ and  \eqref{condition_dif_psi1}.
For the second assertion, we apply the Taylor's expansion theorem twice to get
\begin{multline*}
 \psi(X_{T}^{m^\ell}) - \psi(X_T^{m^{\ell-1}}) = \nabla \psi(X_T). (X_{T}^{m^\ell}-X_T^{m^{\ell-1}}) \\
 + (X_{T}^{m^\ell} - X_{T}) \varepsilon(X_{T},X_{T}^{m^{\ell}}-X_T) -(X_{T}^{m^{\ell-1}} - X_{T}) \varepsilon(X_{T},X_{T}^{m^{\ell-1}}-X_T).
\end{multline*}
The function $\varepsilon$ is given by the Taylor-young expansion, so it satisfies $ \varepsilon({X}_{T}, X_{T}^{m^\ell}-X_{T})\overset{\mathbb{P}}{\underset{\ell \rightarrow \infty}{\longrightarrow}} 0$ and $ \varepsilon({X}_{T}, X_{T}^{m^{\ell-1}}-X_{T})\overset{\mathbb{P}}{\underset{\ell \rightarrow \infty} \longrightarrow} 0$.
By property $(\bf \mathcal P)$, we get the tightness of $r_\ell \ (X_{T}^{m^\ell}-X_{T})$ and $r_\ell \ (X_{T}^{m^{\ell-1}}-X_{T})$ and we deduce
\begin{equation*}
 r_\ell \ \left( (X_{T}^{m^\ell} - X_{T}) \varepsilon(X_{T},X_{T}^{m^{\ell}}-X_T) -(X_{T}^{m^{\ell-1}} - X_{T}) \varepsilon(X_{T},X_{T}^{m^{\ell-1}}-X_T) \right) \overset{\mathbb{P}}{\underset{\ell \rightarrow \infty}{\longrightarrow}} 0.
\end{equation*}
So, according to the stable convergence theorem, we conclude that
\begin{equation*}
 r_\ell \ \left(\psi(X_{T}^{m^\ell})-\psi(X_{T}^{m^{\ell-1}}) \right) \overset{\rm stably}\Longrightarrow \nabla \psi(X_T).U_T,  \quad \mbox{as } \ell \rightarrow \infty.
\end{equation*}

\subsection{Extended Version of the Toeplitz Lemma}
We recall the double indexed version of the Toeplitz lemma for a proof see Lemma 4.1 in \cite{BenalayaHajjiKebaier}. 
\begin{lemma}\label{Toep:ex}
\label{toeplitz}
Let $(a_{i})_{  1\leq i \leq k_{n} } $ a sequence of real positive numbers, where $k_{n}\uparrow \infty$   as $n$ tends to infinity, and $(x_{i}^{n})_{i\geq 1,n\geq 1}$ a  double indexed sequence such that
\begin{itemize}
 \item[(i)] $\lim\limits_{n \to \infty} \sum_{ 1\leq i \leq k_{n}} a_{i}=\infty$
 \item[(ii)]  $\lim\limits_{i,n \to \infty}x_{i}^{n} =\lim\limits_{i \to \infty}(\lim\limits_{n \to \infty}x_{i}^{n})= \lim\limits_{n \to \infty}(\lim\limits_{i \to \infty}x_{i}^{n})=x<\infty$
\end{itemize}
Then
$$\lim\limits_{n\rightarrow +\infty} \frac{\sum_{i=1}^{k_{n}} a_{i} x_{i}^{n}}{\sum_{i=1}^{k_{n}} a_{i}} =x.$$
\end{lemma}
\subsection{Lindeberg Feller Central Limit Theorem for independent random variables}
We recall first the Lindeberg Feller central limit theorem for independent random variables.

\begin{theorem}[Lindeberg Feller Central Limit Theorem \cite{Bil}]
\label{CLT Lindeberg Feller}
Let $(k_n)_{n\in\NN}$ be a sequence such that $k_n \longrightarrow \infty$, as $n \longrightarrow \infty$ and for each $n\in \NN$ we consider a sequence $X_{n1}, X_{n2}, ..., X_{n k_{n}}$ of independent centered and real square integrable random variables. We make the following two assumptions.
\begin{itemize}
 \item [${\it A1.}$] There exists a positive constant $v$ such that
 $\sum_{i=1}^{k_n} \EE (X_{ni})^2 \underset{n\rightarrow\infty}{\longrightarrow} v $.
 \item [${\it A2.}$] Lindeberg's condition holds: that is for all $\varepsilon >0$,
  $\sum_{i=1}^{k_n} \EE (|X_{ni}|^2 \mathbf{1}_{|X_{ni}|\geq \varepsilon}) \underset{n\rightarrow\infty}{\longrightarrow} 0$.
 Then 
 $$\sum_{i=1}^{k_n} X_{ni} \xrightarrow{\mathcal{L}} \mathcal{N}(0,v)\quad\mbox{ as }\; n\rightarrow\infty. $$ 
\end{itemize}
\end{theorem}
\paragraph{Remark.}
The following assumption  known as the Lyapunov condition implies the Lindeberg's condition {\it A2.}
\begin{itemize}
\item [${\it A3.}$] There exists a real number $a>1$ such that
$$ \sum_{k=1}^{k_{n}} \mathbb{E} \left[|X_{ni}|^{2a} \right] \underset{n\rightarrow\infty}{\longrightarrow} 0.$$
\end{itemize}

\end{document}